\documentclass[aos,preprint]{imsart}

\RequirePackage[OT1]{fontenc}
\RequirePackage{amsthm,amsmath,natbib}
\RequirePackage[colorlinks,citecolor=blue,urlcolor=blue]{hyperref}

\usepackage{graphicx}

\setattribute{journal}{name}{}

\startlocaldefs
\numberwithin{equation}{section}
\theoremstyle{plain}

\endlocaldefs

\RequirePackage[colorlinks]{hyperref}
\RequirePackage{hypernat}

\usepackage[psamsfonts]{amsfonts}

\usepackage{pdflscape}

\usepackage{dsfont}
 
\defcitealias{Cut2015}{CPV15}

\newtheorem{Lem}{Lemma}[section]
\newtheorem{Prop}{Proposition}[section]
\newtheorem{Theor}{Theorem}[section]

\newcommand{\cqfd}{\hfill $\square$}

\newcommand{\R}{\mathbb R}

\newcommand{\n}{^{(n)}}

\newcommand{\Xb}{\mathbf{X}}
\newcommand{\Sb}{\mathbf{S}}

\newcommand{\Vb}{\mathbf{V}}

\newcommand{\Tb}{\ensuremath{\mathbf{T}}}

\newcommand{\xb}{\ensuremath{\mathbf{x}}}

\newcommand{\Ob}{\ensuremath{\mathbf{O}}}

\newcommand{\Ub}{\ensuremath{\mathbf{U}}}

\newcommand{\thetab}{{\pmb \theta}}

\newcommand{\Sigb}{{\pmb \Sigma}}

\newcommand{\Deltab}{{\pmb \Delta}}
\newcommand{\taub}{{\pmb \tau}}

\newcommand{\betab}{{\pmb \beta}}

\newcommand{\pr}{^{\prime}}

\newcommand{\ny}{n\rightarrow\infty}

\begin{document}

\begin{frontmatter}
\title{Testing Uniformity on High-Dimensional Spheres 
\mbox{against 
Monotone 
Rotationally Symmetric Alternatives}}
\runtitle{Testing Uniformity on High-Dimensional Spheres}

\begin{aug}
\author{\fnms{Christine} \snm{Cutting}\ead[label=e1]{Christine.Cutting@ulb.ac.be}},
\author{\fnms{Davy} \snm{Paindaveine}\thanksref{t2}\ead[label=e2]{dpaindav@ulb.ac.be}
\ead[label=u2,url]{http://homepages.ulb.ac.be/\textasciitilde dpaindav}}
\and
\author{\fnms{ Thomas} \snm{Verdebout}
\ead[label=e3]{tverdebo@ulb.ac.be}
\ead[label=u3,url]{http://tverdebo.ulb.ac.be}}

\thankstext{t2}{Research is supported by an \mbox{A.R.C.} contract from the Communaut\'{e} Fran\c{c}aise de Belgique and by the IAP research network grant \mbox{nr.} P7/06 of the Belgian government (Belgian\vspace{1mm}
Science Policy).}
\runauthor{Chr. Cutting, D. Paindaveine, and Th. Verdebout}

\affiliation{Universit\'{e} libre de Bruxelles}

\address{Universit\'{e} libre de Bruxelles\\
D\'{e}partement de Math\'{e}matique\\
Bld du Triomphe, 
Campus Plaine, CP210\\
B-1050, Brussels\\
Belgium\\
\printead{e1}\\
\phantom{E-mail:\ }\printead*{e3}\\
\printead{u3}}

\address{Universit\'{e} libre de Bruxelles\\
ECARES and D\'{e}partement de Math\'{e}matique\\
Avenue F.D. Roosevelt, 50\\
ECARES, CP114/04\\
B-1050, Brussels\\
Belgium\\
\printead{e2}\\
\printead{u2}}
\end{aug}

\vspace{4mm}
\begin{abstract}
We consider the problem of testing uniformity on high-dimen\-sional unit spheres. We are primarily interested in non-null issues. We show that rotationally symmetric alternatives lead to two Local Asymptotic Normality (LAN) structures. The first one is for fixed modal location~$\thetab$ and allows to derive locally asymptotically most powerful tests under specified~$\thetab$. The second one, that addresses the Fisher--von Mises--Langevin (FvML) case, relates to the unspecified-$\thetab$ problem and shows that the high-dimensional Rayleigh test is locally asymptotically most powerful invariant. Under mild assumptions, we derive the asymptotic non-null distribution of this test, which allows to extend away from the FvML case the asymptotic powers obtained there from Le Cam's third lemma. Throughout, we allow the dimension~$p$ to go to infinity in an arbitrary way as a function of the sample size~$n$. Some of our results also strengthen the local optimality properties of the Rayleigh test in low dimensions. We perform a Monte Carlo study to illustrate our asymptotic results. Finally, we treat an application related to testing for sphericity in high dimensions. 
\vspace{-1mm}
\end{abstract}



\end{frontmatter}


\section{Introduction}

In directional statistics, inference is based on $p$-variate observations lying on the unit sphere~$\mathcal{S}^{p-1}:=\{\xb\in\R^{p} : \|\xb\|=\sqrt{\xb'\xb}=1\}$. This is relevant in various situations. (i) First,  the original data themselves may belong to~$\mathcal{S}^{p-1}$; classical examples involve wind direction data ($p=2$) or spatial data at the earth scale $(p=3)$. 
\linebreak
(ii) Second, some fields by nature are so that only the relative magnitude of the observations is important, which leads to projecting observations onto~$\mathcal{S}^{p-1}$. In shape analysis, for instance, this projection only gets rid of an overall scale factor related to the (irrelevant) object size. (iii) Finally, even in inference problems where the full (Euclidean) observations in principle need to be considered, a common practice in nonparametric statistics is to restrict to sign procedures, that is, to procedures that are measurable with respect to the projections of the observations onto~$\mathcal{S}^{p-1}$; see, e.g.,  \cite{Oja2010} and the references therein. 

While~(i) is obviously restricted to small dimensions~$p$,~(ii)-(iii) nowadays increasingly involve high-dimensional data. For~(ii), high-dimensional directional data were considered in~\cite{Dry2005}, with applications in brain shape modeling; in text mining, \cite{banerjee2003generative} and \cite{BanGho2004} project high-dimensional data on unit spheres to discard text sizes when performing clustering. As for~(iii), the huge interest raised by high-dimensional statistics in the last decade has made it natural to consider high-dimensional \emph{sign} tests. In particular, \cite{Zouetal2013} 
recently considered the high-dimensional version of the \cite{HP06} sign tests of sphericity, whereas an extension to the high-dimensional case of the location sign test from \cite{Cha1992} and \cite{mooj95} was recently proposed in \cite{Runze2015}. Considering~(iii) in high dimensions is particularly appealing since for moderate-to-large~$p$, sign tests show excellent (fixed-$p$) efficiency properties (see \cite{PaiVer2015} for details). Also, the concentration-of-measure phenomenon may make the restriction to signs virtually void as the dimension~$p$ increases. 

In this paper, we consider the problem of testing uniformity on the unit sphere~$\mathcal{S}^{p-1}$, both in low and high dimensions. In low dimensions, this is a fundamental problem that has been extensively treated; see \cite{MarJup2000} and the references therein. 
The high-dimensional version of the problem is less standard, yet also has some history. \cite{cueetal2009} proposed a test of uniformity that performs well empirically even in high dimensions, but no asymptotic results were obtained as~$p$ goes to infinity. \cite{Chi1991,Chi1993} explicitly considered high-dimensional testing for uniformity on the sphere, 
 in a fixed-$n$ large-$p$ framework, while  \cite{Caietal2013} rather adopted a double asymptotic approach for the same problem. Possible applications of testing uniformity on high-dimensional spheres include outlier detection; see \cite{JuaPri2001}. Other natural applications are related with testing for sphericity in~$\R^p$, in the spirit of~(iii) above; in Section~\ref{realsec}, we will elaborate on this and provide references.

To be more specific, assume that the observations form of a triangular array of random vectors $\Xb_{ni}$, $i=1,\ldots,n$, $n=1,2,\ldots,$ where, for any~$n$, the~$\Xb_{ni}$'s are mutually independent and share a common distribution on the unit sphere~$\mathcal{S}^{p_n-1}$, and consider the problem of testing the null hypothesis~$\mathcal{H}_{0n}$ that this common distribution is the uniform over~$\mathcal{S}^{p_n-1}$. While our main interest is in the high-dimensional case ($p_n\to\infty)$, most of our results will also address the (low-dimensional) classical fixed-$p$ case ($p_n=p$ for all~$n$). The most classical test of uniformity is the \cite{Ray1919} 
test, that rejects~$\mathcal{H}_{0n}$ for large values of
$
R_n
:= 
np_n \| \bar{\Xb}_n\|^2
,
$ 
where $\bar{\Xb}_n:=\frac{1}{n} \sum_{i=1}^n \Xb_{ni}$. For fixed~$p$, the test is based on the null asymptotic $\chi^2_p$ distribution of~$R_n$. In the high-dimensional setup, \cite{PaiVer2015} obtained the following asymptotic normality result under the null.

\begin{Theor}
\label{raylnull}
Let $(p_n)$ be a sequence of positive integers diverging to~$\infty$ as $\ny$. 
Assume that the triangular array $\Xb_{ni}$, $i=1,\ldots,n$, $n=1,2,\ldots,$ is such that, for any $n$, $\Xb_{n1},\Xb_{n2},\ldots,\Xb_{nn}$ form a random sample from the uniform distribution on~$\mathcal{S}^{p_n-1}$. Then  
\vspace{1mm}
\begin{equation}
\label{raylhdstat}
{R}_n^{\rm St}
:=
\frac{R_n-p_n}{\sqrt{2 p_n}}
=
\frac{\sqrt{2p_n}}{n} 
\sum_{1\leq i< j\leq n}
 \Xb_{ni}\pr \Xb_{nj}
\,
\stackrel{\mathcal{D}}{\to}
\,
\mathcal{N}(0,1)
\vspace{-2mm}
\end{equation}
as~$n\to\infty$, where~$\stackrel{\mathcal{D}}{\to}$ denotes weak convergence.
\end{Theor}

\newpage
Denoting by~$\Phi(\cdot)$ the cumulative distribution function of the standard normal, the high-dimensional Rayleigh test ($\phi\n$, say) then rejects~$\mathcal{H}_{0n}$ at asymptotic level~$\alpha$ whenever
\begin{equation}
\label{RayHDtraindelavie}
{R}_n^{\rm St}
>
z_\alpha,
\quad
\textrm{ with } 
z_\alpha:=\Phi^{-1}(1-\alpha).
\end{equation}
Remarkably, this test does not impose any condition on the way $p_n$ goes to infinity with~$n$, hence can be applied as soon as~$n$ and~$p_n$ are large, without bothering about their relative magnitude (in contrast, most results in high-dimensional statistics typically impose that~$p_n/n\to c$ for some~$c>0$). Theorem~\ref{raylnull}, however, is not sufficient to justify resorting to the Rayleigh test~: the trivial test, that would discard the data and reject~$\mathcal{H}_{0n}$ with probability~$\alpha$, has indeed the same asymptotic null behaviour as the high-dimensional Rayleigh test, yet has a power function that is uniformly equal to the nominal level~$\alpha$. One of the main goals of this paper is to study the non-null behaviour of the Rayleigh test and to show that this test actually enjoys nice optimality properties, both in the low- and high-dimensional cases. Optimality throughout will be in the Le Cam sense, in relation with the Local Asymptotic Normality (LAN) structures of the models we adopt below.

The outline of the paper is as follows. In Section~\ref{contiguitysec}, we define a class of alternatives to the null of uniformity that skew the probability mass along a ``modal direction"~$\thetab$, and we identify the corresponding contiguous alternatives. In Section~\ref{oraclesec}, we provide a LAN result for fixed~$\thetab$, which leads to locally asymptotically most powerful tests under specified~$\thetab$. We address the \mbox{un\mbox{specified-$\thetab$}} problem through invariance arguments in Section~\ref{invariantsec}, which, in the FvML case, provides a second LAN result and shows that the high-dimensional Rayleigh test is locally asymptotically most powerful invariant. In Section~\ref{Rayleighsec}, we derive the asymptotic distribution of the high-dimensional Rayleigh test under general rotationally symmetric alternatives and comment on the resulting limiting powers. In Section~\ref{simusec}, we illustrate our asymptotic results through simulations. In Section~\ref{realsec}, we link the problem considered to that of testing for sphericity in high dimensions and we treat a real data example. In Section~\ref{conclusec}, we summarize the main findings of the paper and discuss some perspectives for future reseach. Finally, the appendix and the supplementary article \cite{Cut2015} collect technical proofs.

\section{Contiguous rotationally symmetric alternatives}
\label{contiguitysec}

Throughout, we consider specific alternatives to the null of uniformity over the $p$-dimensional unit sphere~$\mathcal{S}^{p-1}$, namely rotationally symmetric alternatives. A $p$-dimensional vector~$\Xb$ is said to be \emph{rotationally symmetric about}~$\thetab(\in\mathcal{S}^{p-1})$ if and only if $\mathbf{O}\Xb$ is equal in distribution to~$\Xb$ for any orthogonal $p\times p$ matrix~$\mathbf{O}$ satisfying $\mathbf{O}\thetab=\thetab$; see, e.g., \cite{saw1978}. Such distributions are fully characterized by the location parameter~$\thetab$ and the cumulative distribution function~$F$ of~$\Xb'\thetab$. The null of uniformity (under which~$\thetab$ is not identifiable) is obtained for
\begin{equation}
\label{unifcdf}
F_p(t) 
:=
c_p
\int_{-1}^t (1-s^2)^{(p-3)/2}\,ds 
,
\ 
\textrm{ with }
c_{p}
:=
\frac{\Gamma\big(\frac{p}{2}\big)}{\sqrt{\pi}\,\Gamma\big(\frac{p-1}{2}\big)}
,
\end{equation}
where~$\Gamma(\cdot)$ is the Euler Gamma function.
Particular alternatives are given, e.g., by the so-called Fisher--von Mises--Langevin (FvML) distributions, that correspond to
 \begin{equation}
\label{fvmlcdf}
F_{p,\kappa}^{\rm FvML}(t)
:=
c^{\rm FvML}_{p,\kappa}\int_{-1}^t (1-s^2)^{(p-3)/2}\exp(\kappa s)\,ds 
, 
\ 
\textrm{ with }
c^{\rm FvML}_{p,\kappa}
:=
\frac{(\kappa/2)^{\frac{p}{2}-1}}{\sqrt{\pi}\,\Gamma(\frac{p-1}{2})\mathcal{I}_{\frac{p}{2}-1}(\kappa)}
,
\end{equation}
where~$\mathcal{I}_\nu(\cdot)$ is the order-$\nu$ modified Bessel function of the first kind and~$\kappa(>0)$ is a \emph{concentration} parameter (the larger the value of~$\kappa$, the more concentrated about~$\thetab$ the distribution is); see \cite{MarJup2000} for further details.  

In Sections~\ref{contiguitysec} to~\ref{invariantsec}, we actually restrict to ``monotone" rotationally symmetric densities (with respect to the surface area measure on~$\mathcal{S}^{p-1}$) of the form
\begin{equation}
\label{densconc}
\xb\mapsto c_{p,\kappa,f} f(\kappa\, \xb'\thetab)
,
\qquad
\xb\in\mathcal{S}^{p-1},
\end{equation}
where~$\thetab(\in\mathcal{S}^{p-1})$ is a location parameter,~$\kappa(>0)$ is a concentration parameter, and the function~$f:\R\to\R^+$ is monotone strictly increasing, differentiable at~$0$, and satisfies~$f(0)=f'(0)=1$. These conditions on~$f$, that will be tacitly assumed throughout, guarantee identifiability of~$\thetab$,~$\kappa$ and~$f$~: clearly, the strict monotonicity of~$f$ implies that~$\thetab$ is the modal location on~$\mathcal{S}^{p-1}$, whereas the constraint~$f'(0)=1$ allows to identify~$\kappa_n$ and~$f$. 
%
Note that irrespective of~$f$, the boundary value~$\kappa=0$ corresponds to the uniform distribution over~$\mathcal{S}^{p-1}$. It is well-known that, if~$\Xb$ has density~(\ref{densconc}), then~$\Xb'\thetab$ has density  
$
t
\mapsto
c_{p,\kappa,f}  (1-t^2)^{(p-3)/2} f(\kappa t)\,\mathbb{I}[t\in[-1,1]]
$
(throughout~$\mathbb{I}[A]$ stands for the indicator function of the set or condition~$A$). This is compatible with the cumulative distribution functions in~(\ref{unifcdf})-(\ref{fvmlcdf}), and shows that
$
c_{p,\kappa,f}
=
1/\big( \int_{-1}^1 (1-t^2)^{(p-3)/2} f(\kappa t)\,dt \big)
.
$
Finally, note that~$f(\cdot)=f_{\rm FvML}(\cdot)=\exp(\cdot)$ provides the FvML distributions above.

As announced in the introduction, we consider triangular arrays of observations~$\Xb_{ni}$, $i=1,\ldots,n$, $n=1,2,\ldots$ where the random vectors~$\Xb_{ni}$, $i=1,\ldots,n$ take values in~$\mathcal{S}^{p_n-1}$. More specifically, for any~$\thetab_n\in\mathcal{S}^{p_n-1}$, $\kappa_n>0$ and~$f$ as above, we will denote as~${\rm P}\n_{\thetab_n,\kappa_n,f}$ the hypothesis under which~$\Xb_{ni}$, $i=1,\ldots,n$ are mutually independent and share the common density~$\xb\mapsto c_{p_n,\kappa_n,f} f(\kappa_n\, \xb'\thetab_n)$. Note that larger values of~$\kappa_n$ provide  increasingly severe deviations from the null of uniformity, which is obtained as~$\kappa_n$ goes to zero. Denoting the null hypothesis as~${\rm P}\n_{0}$, it is then natural to wonder whether or not  ``appropriately small" sequences~$\kappa_n$ make~${\rm P}\n_{\thetab_n,\kappa_n,f}$ and~${\rm P}\n_{0}$ mutually contiguous. The following result answers this question (see Appendix~\ref{appA} for a proof). 
\vspace{1mm}

\begin{Theor}
\label{contigtheor}
Let~$(p_n)$ be a sequence in~$\{2,3,\ldots\}$. Let $(\thetab_n)$ be a sequence such that $\thetab_n\in\mathcal{S}^{p_n-1}$ for all~$n$, $(\kappa_n)$ be a positive sequence such that $\kappa_n^2=O(\frac{p_n}{n})$, and assume that $f$ is twice differentiable at~$0$. Then, the sequence of alternative hypotheses~${\rm P}\n_{\thetab_n,\kappa_n,f}$ and the null sequence~${\rm P}\n_{0}$ are mutually contiguous.
\end{Theor}

This contiguity result covers both the low- and high-dimensional cases. In the low-dimensional case, the usual parametric rate $\kappa_n\sim 1/\sqrt{n}$ provides contiguous alternatives, which implies that, irrespective of~$f$, there exist no consistent tests for~$\mathcal{H}_{0n}: \{{\rm P}\n_{0}\}$ against~$\mathcal{H}_{1n}: \{{\rm P}\n_{\thetab_n,\kappa_n,f}\}$ if $\kappa_n=\tau/\sqrt{n}$, $\tau>0$. The high-dimensional case is more interesting. First, we stress that the contiguity result in Theorem~\ref{contigtheor} does not impose conditions on~$p_n$, hence in particular applies when~(a) $p_n/n\to c$ for some~$c>0$ or (b) $p_n/n\to\infty$. Interestingly, the result shows that contiguity in cases~(a)-(b) can be achieved for sequences~($\kappa_n$) that do not converge to zero~: a constant sequence~($\kappa_n$) ensures contiguity in case~(a), whereas contiguity in case~(b) may even be obtained for a sequence ($\kappa_n$) that diverges to infinity in a suitable way. In both cases, there then exist no consistent tests for~$\mathcal{H}_{0n}: \{{\rm P}\n_{0}\}$ against the corresponding sequences of alternatives~$\mathcal{H}_{1n}: \{{\rm P}\n_{\thetab_n,\kappa_n,f}\}$, despite the fact that the sequences~$(\kappa_n)$ are not~$o(1)$. This may be puzzling at first since such sequences are expected to lead to severe alternatives to uniformity; it actually makes sense, however, that the fast increase of the dimension~$p_n$, despite the favorable sequences~$(\kappa_n)$, makes the problem difficult enough to prevent the existence of consistent tests.

\section{Optimal testing under specified modal location}
\label{oraclesec}

Whenever the modal location~$\thetab_n$ is specified (a case that is explicitly treated in~\citealp{MarJup2000}), optimal tests of uniformity can be obtained from the following Local Asymptotic Normality (LAN) result (see Appendix~\ref{appA}  for a proof). To the best of our knowledge, this result provides the first instance of the LAN structure in high dimensions. 

\begin{Theor}
\label{LANtheor}
Let~$(p_n)$ be a sequence in~$\{2,3,\ldots\}$ and let $(\thetab_n)$ be a sequence such that $\thetab_n\in\mathcal{S}^{p_n-1}$ for all~$n$. Let $\kappa_n=\tau_n\sqrt{p_n/n}$, where the positive sequence~$(\tau_n)$ is~$O(1)$ but not~$o(1)$, and assume that $f$ is twice differentiable at~$0$. 
%
Then, as~$\ny$ under~${\rm P}_{0}\n$, 
\begin{equation}
\label{LAQmain}
\log \frac{d{\rm P}\n_{\thetab_n,\kappa_n,f}}{d{\rm P}\n_{0}} 
=
\tau_n
\Delta_{\thetab_n}\n
-
\frac{\tau_n^2}{2} 
+
o_{\rm P}(1)
,
\end{equation}
where $
\Delta_{\thetab_n}\n
:=
\sqrt{n p_n}\,
\bar{\Xb}_n\pr \thetab_n
$
is asymptotically standard normal. 
In other words,  the model $\{ {\rm P}\n_{\thetab_n,\kappa,f} : \kappa\geq 0 \}$ (where~${\rm P}\n_{\thetab_n,0,f}:={\rm P}\n_0$ for any~$\thetab_n$ and~$f$) is locally asymptotically normal at~$\kappa=0$ with central sequence~$\Delta_{\thetab_n}\n$, Fisher information~$1$, and contiguity rate~$\sqrt{p_n/n}$.
\end{Theor}

This result, that covers both the low- and high-dimensional cases, reveals that the rate~$\kappa_n\sim \sqrt{p_n/n}$ in Theorem~\ref{contigtheor} is actually the contiguity rate of the considered model (that is, more severe alternatives are not contiguous to the null of uniformity). In low dimensions, the usual parametric contiguity rate $\kappa_n\sim 1/\sqrt{n}$ is obtained. The high-dimensional rate is of course non-standard. Yet in the FvML high-dimensional case, this rate may be related to the fact that, as~$p\to\infty$, one needs to consider~$\kappa_p\sim\sqrt{p}$ to obtain FvML $p$-vectors that provide non-degenerate weak limiting results that are different from those obtained from $p$-vectors that are uniform over the sphere (see \cite{Wat1988} for a precise result); the contiguity rate~$\kappa_n\sim \sqrt{p_n/n}$ then intuitively results from a standard $1/\sqrt{n}$-shrinkage starting from this non-trivial~$\kappa_p\sim\sqrt{p}$ high-dimensional situation. 

Now, consider the \mbox{specified-$\thetab_n$} problem, that is, the problem of testing~$\{{\rm P}\n_{0}\}$ (uniformity over~$\mathcal{S}^{p_n-1}$) against~$\cup_{\kappa>0} \cup_f \{{\rm P}\n_{\thetab_n,\kappa,f}\}$. Theorem~\ref{LANtheor} entails that the test~$\phi_{\thetab_n}\n$ rejecting the null at asymptotic level~$\alpha$ whenever
\begin{equation}
\label{deforayoup}
\Delta_{\thetab_n}\n
=
\sqrt{n p_n}
\,
\bar{\Xb}_n\pr \thetab_n
>
z_\alpha
\end{equation}
is \emph{locally asymptotically most powerful}. 
Since Le Cam's third lemma readily implies that $\Delta_{\thetab_n}\n$ is asymptotically normal with mean~$\tau$ and variance one  under~${\rm P}\n_{\thetab_n,\kappa_n,f}$, with~$\kappa_n=\tau\sqrt{p_n/n}$, the corresponding asymptotic power of~$\phi_{\thetab_n}\n$ is
\begin{equation}
\label{poworacle}
\lim_{n\to\infty}
{\rm P}\n_{\thetab_n,\kappa_n,f}
\big[
\Delta_{\thetab_n}\n > z_\alpha
\big]
=
1-\Phi(z_\alpha - \tau)
.
\end{equation}
While all results of this section so far covered both the low- and high-dimensional cases, we need to treat these cases separately to investigate how the Rayleigh test compares with the optimal test~$\phi_{\thetab_n}\n$.  
  
We start with the low-dimensional case. Denoting by~$\chi^2_p(\delta)$ the non-central chi-square distribution with~$p$ degrees of freedom and non-centrality parameter~$\delta$, Le Cam's third lemma allows to show that, under the contiguous alternatives~${\rm P}\n_{\thetab_n,\kappa_n,f}$, with~$\kappa_n=\tau\sqrt{p/n}$ (compare with the local alternatives from Theorem~\ref{LANtheor}),
\begin{equation}
\label{fixedplaw}
R_n
\stackrel{\mathcal{D}}{\to}
\chi^2_p(\tau^2)
\end{equation}
as~$\ny$; for the sake of completeness, we provide a proof in the supplementary article \cite{Cut2015}. Denoting by~$\Psi_p(\cdot)$ the cumulative distribution function of the $\chi^2_p$ distribution,  the corresponding asymptotic power of the Rayleigh test is therefore
\begin{equation}
\label{powfixedp}
\lim_{n\to\infty}
{\rm P}\n_{\thetab_n,\kappa_n,f}
\big[
R_n
>
\Psi^{-1}_p(1-\alpha)
\big]
=
{\rm P}
\big[
Y>\Psi^{-1}_p(1-\alpha)
\big],
\quad
\textrm{ with } 
Y\sim \chi^2_p(\tau^2)
,
\end{equation}
which is strictly smaller than the asymptotic power in~(\ref{poworacle}). We conclude that, in the \mbox{specified-$\thetab_n$} case, the low-dimensional Rayleigh test is not locally asymptotically most powerful  yet shows non-trivial asymptotic powers against contiguous alternatives.

The story is different in the high-dimensional case, as it can be guessed from the following heuristic reasoning. In view of~(\ref{fixedplaw}), we have that, as~$\ny$ under ${\rm P}\n_{\thetab_n,\kappa_n,f}$, with~$\kappa_n=\tau\sqrt{p/n}$, 
$$
R_n^{\rm St}
=
\frac{R_n-p}{\sqrt{2p}}
\stackrel{\mathcal{D}}{\to}
\frac{\chi^2_1(\tau^2)-1}{\sqrt{2p}} + 
\frac{\chi^2_{p-1}-(p-1)}{\sqrt{2p}}
,
$$
where both chi-square variables are independent. When both~$n$ and~$p$ are large, it is therefore expected that, under the same sequence of alternatives, 
 $
R_n^{\rm St}
\approx
\mathcal{N}\big(\frac{\tau^2}{\sqrt{2p}}
\,
,
 1 + \frac{2\tau^2}{p}\big)
,
$
where~$Z_n\approx \mathcal{L}$ means that the distribution of~$Z_n$ is close to~$\mathcal{L}$. Thus, in the high-dimensional case (where~$p=p_n\to\infty$), $R_n^{\rm St}$ is expected to be standard normal under these alternatives, which would imply that the high-dimensional Rayleigh test in~(\ref{RayHDtraindelavie}) has asymptotic powers equal to the nominal level~$\alpha$. 

The high-dimensional LAN result in Theorem~\ref{LANtheor} allows to confirm these heuristics. Letting~$\kappa_n=\tau_n\sqrt{p_n/n}$, where~$\tau_n$ is $O(1)$, Theorem~\ref{LANtheor} readily yields that, as~$\ny$,
\begin{eqnarray*}
\lefteqn{
{\rm Cov}_{{\rm P}_{0}\n}\!
\bigg[
{R}_n^{\rm St}
,
\log \frac{d{\rm P}\n_{\thetab_n,\kappa_n,f}}{d{\rm P}\n_{0}} 
\bigg]
=
{\rm Cov}_{{\rm P}_{0}\n}\!
\big[{R}_n^{\rm St},\Delta_{\thetab_n,f}\n\big]
\tau_n
+o(1)
}
\\[2mm]
& & 
\hspace{10mm}
=
\frac{\sqrt{2}p_n}{n^{3/2}} \tau_n
\sum_{i=1}^n 
\sum_{1\leq k<\ell\leq n}
{\rm E}_{{\rm P}_{0}\n}\!
[(\Xb_{ni}\pr \thetab_n) (\Xb_{nk}\pr \Xb_{n\ell})]
+
o(1)
=
o(1)
,
\end{eqnarray*}
so that Le Cam's third lemma implies that $R_n^{\rm St}$ remains asymptotically standard normal under~${\rm P}\n_{\thetab_n,\kappa_n,f}$. 
This confirms that, unlike in the low-dimensional case, the high-dimensional Rayleigh test does not show any power under the contiguous alternatives from Theorem~\ref{LANtheor}. In other words, the high-dimensional Rayleigh test fails to be rate-consistent for the \mbox{specified-$\thetab_n$} problem. 

The Rayleigh test, however, does not make use of the specified value of the modal location~$\thetab_n$, hence does not primarily address the \mbox{specified-$\thetab_n$} problem but rather the \mbox{unspecified-$\thetab_n$} one. Therefore, the key question is whether or not the Rayleigh test is optimal for the latter problem. We answer this question in the next section.

\section{Optimal testing under unspecified modal location}
\label{invariantsec}

Building on the results of the previous section, two natural approaches, that may lead to an optimal test for the \mbox{\mbox{un\mbox{specified-$\thetab_n$}}} problem, are the following. The first one consists in substituting an estimator~$\hat\thetab_n$ for~$\thetab_n$ in the optimal test~$\phi_{\thetab_n}\n$ above. For the so-called \emph{spherical mean}~$\hat\thetab_n=\bar{\Xb}_n/\|\bar{\Xb}_n\|$ (which is the MLE for~$\thetab_n$ in the FvML case), the resulting test rejects the null for large values of
$
\Delta_{\hat\thetab_n}\n
=
\sqrt{n p_n}
\,
\bar{\Xb}_n\pr \hat\thetab_n
=
\sqrt{n p_n}
\,
\|\bar{\Xb}_n\|
=R_n^{1/2}
,
$
hence coincides with the Rayleigh test. The second approach, in the spirit of~\cite{Dav1977,Dav1987,Dav2002}, rather consists in adopting the test statistic
$
\sup_{\thetab_n\in\mathcal{S}^{p_n-1}}
\Delta_{\thetab_n}\n
=
\sqrt{n p_n}
\,
\|\bar{\Xb}_n\|
,
$
which again leads to the Rayleigh test. These considerations suggest that the Rayleigh test indeed may be optimal for the \mbox{un\mbox{specified-$\thetab_n$}} problem. In this section, we investigate whether this is the case or not, both in low and high dimensions.

\subsection{The low-dimensional case}
\label{invariantlowsec}

To investigate the optimality properties of the low-dimensional Rayleigh test for the \mbox{un\mbox{specified-$\thetab_n$}} problem, it is helpful to adopt a new parametri\-zation. For fixed~$p$ and~$f$, the model is indexed by $(\thetab,\kappa)\in\mathcal{S}^{p-1}\times\R^+$, where the value~$\kappa=0$ makes~$\thetab$ unidentified (for fixed~$p$, the dimension of~$\thetab$ does not depend on~$n$, so that there is no need to consider sequences~$(\thetab_n)$). We then consider the alternative parametrization in~${\pmb \mu}:=\kappa\thetab$, for which the fixed-$p$ result in Theorem~\ref{LANtheor} readily rewrites as follows.

\begin{Theor}
\label{LANtheorlow}
Fix an integer~$p\geq 2$ and let ${\pmb\mu}_n=\sqrt{p/n}\,\taub_n$ for all~$n$, where the sequence~$(\taub_n)$ in~$\R^p$ is~$O(1)$ but not~$o(1)$. Assume that $f$ is twice differentiable at~$0$. For any~${\pmb \mu}\in\R^p\setminus\{\mathbf{0}\}$, let ${\rm P}\n_{{\pmb\mu},f}:={\rm P}\n_{\thetab,\kappa,f}$, where~${\pmb \mu}=:\kappa\thetab$, with~$\thetab\in\mathcal{S}^{p-1}$.
Then, as~$\ny$ under~${\rm P}_{0}\n$, 
$
\log \big(d{\rm P}\n_{{\pmb\mu}_n,f}/d{\rm P}\n_{0}\big)
=
\taub_n\pr
\Deltab\n
-
\frac{1}{2} \|\taub_n\|^2
+
o_{\rm P}(1)
$, where $\Deltab\n:=\sqrt{n p}\,\bar{\Xb}_n$
is asymptotically standard $p$-variate normal. 
\end{Theor}

In the new parametrization, note that the problem of testing uniformity consists in testing~$\mathcal{H}_0\n: {\pmb\mu}={\bf 0}$ versus~$\mathcal{H}_1\n: {\pmb\mu}\neq {\bf 0}$. Theorem~\ref{LANtheorlow} then ensures that the test rejecting the null at asymptotic level~$\alpha$ whenever
$
\|\Deltab\n\|^2=np\| \bar{\Xb}_n \|^2 >\Psi_p^{-1}(1-\alpha)
$ ---
that is, the low-dimensional Rayleigh test --- is locally asymptotically maximin; see, e.g.,  \cite{Lie2008}. This new optimality property of the low-dimensional Rayleigh test complements the one stating that this test is locally most powerful invariant;
see, e.g., \cite{Chi2003}, Section~6.3.5. 

The specified-$\thetab_n$ and unspecified-$\thetab_n$ testing problems are two distinct statistical problems, that, even in the low-dimensional case considered, provide different efficiency bounds. In low dimensions, the Rayleigh test is optimal for the unspecified-$\thetab_n$ problem, but not for the specified-$\thetab_n$ one (the latter suboptimality follows from the fact that the asymptotic powers in~(\ref{powfixedp}) are strictly smaller than those of the optimal test in~(\ref{poworacle})). This thoroughly describes the optimality properties of this test in the low-dimensional case, so that we may now focus on the high-dimensional case.

\subsection{The high-dimensional case}
\label{invarianthighsec}

If~$p_n$ goes to infinity, then the dimension of the parameter~$(\thetab_n,\kappa)$ increases with~$n$, so that there cannot be a high-dimensional analogue of the LAN result in Theorem~\ref{LANtheorlow}. We therefore rather adopt, in the present hypothesis testing context, an \emph{invariance} approach that is close in spirit to the one used by \cite{Mor2009} in a point estimation context. 

The null of uniformity and all collections of alternatives~$\mathcal{P}\n_{\kappa,f}:=\{{\rm P}\n_{\thetab,\kappa,f}:\thetab\in\mathcal{S}^{p_n-1}\}$ (hence also the problem of testing uniformity against rotationally symmetric alternatives itself) are invariant under the group of rotations~$\mathcal{G}\n:=\{g\n_\Ob:\Ob\in SO(p_n)\}$, where~$g\n_\Ob(\xb_1,\ldots,\xb_n)=(\Ob\xb_1,\ldots,\Ob\xb_n)$ for any~$(\xb_1,\ldots,\xb_n)\in \mathcal{S}^{p_n-1}\times \ldots \times \mathcal{S}^{p_n-1}$ \mbox{($n$ times)} and where~$SO(p_n)$ stands for the collection of $p_n\times p_n$ orthogonal matrices with determinant one. 
The invariance principle (see, e.g., \cite{Shao2003}, Section~6.3, or \cite{Lehetal2005}, Chapter~6) then suggests restricting to $\mathcal{G}\n$-invariant tests, that automatically are distribution-free under any~$\mathcal{P}\n_{\kappa,f}$.  

As usual, optimal invariant tests are to be determined in the image of the original model by a maximal invariant~$\Tb_n$ of~$\mathcal{G}\n$. The likelihood (with respect to the surface area measure~$m_{p_n}$ on~$\mathcal{S}^{p_n-1}$) associated with the image of~$\mathcal{P}\n_{\kappa_n,f}$ by~$\Tb_n$ is given by
$$
 \frac{d{\rm P}^{(n)\Tb_n}_{\kappa_n,f}}{dm_{p_n}} 
=
\int_{SO(p_n)}
\prod_{i=1}^n 
\Big[
c_{p_n,\kappa_n,f} 
f(\kappa_n (\Ob\Xb_{ni})\pr \thetab_n)
\Big]
\,d\Ob
,
$$
where the integral is with respect to the Haar measure on~$SO(p_n)$; see, e.g., Lemma~2.5.1 in~\cite{Gir1996}. The resulting log-likelihood ratio to the null of uniformity is therefore
\begin{eqnarray}
\Lambda_{n,f}^{\Tb_n}
\, := \,
\log \frac{d{\rm P}^{(n)\Tb_n}_{\kappa_n,f}}{d{\rm P}\n_{0}} 
&=&
\log 
\,
 \frac{c_{p_n,\kappa_n,f}^n  
\int_{SO(p_n)}
\prod_{i=1}^n 
f(\kappa_n \Xb_{ni} (\Ob\pr \thetab_n))\,d\Ob
}{c_{p_n}^n} 
\nonumber
\\[2mm]
&=&
\log 
\,
 \frac{c_{p_n,\kappa_n,f}^n 
{\rm E}
\big[
\prod_{i=1}^n 
f(\kappa_n \Xb_{ni}\pr \Ub) | \Xb_{n1},\ldots,\Xb_{nn} \big]
}{c_{p_n}^n} 
,
\label{invarL}
\end{eqnarray}
where~$\Ub$ is uniformly distributed over~$\mathcal{S}^{p_n-1}$ and is independent of the~$\Xb_{ni}$'s. The following theorem shows that, in the FvML case~$f(\cdot)=f_{\rm FvML}(\cdot)=\exp(\cdot)$, this collection of log-likelihood ratios enjoys the LAN property.

\begin{Theor}
\label{LANinvartheor}
Let $(p_n)$ be a sequence of positive integers diverging to~$\infty$ as $\ny$ and let~$\kappa_n=\tau_n p_n^{3/4}/\sqrt{n}$, where the positive sequence~$(\tau_n)$ is~$O(1)$ but not~$o(1)$. 
%
Then, as~$\ny$ under~${\rm P}_{0}\n$, we have that 
\begin{equation}
\label{LAQinvarmain}
\log \frac{d{\rm P}^{(n)\Tb_n}_{\kappa_n,f_{\rm FvML}}}{d{\rm P}\n_{0}} 
=
\tau_n^2
\Delta^{(n)\Tb_n}
-
\frac{\tau_n^4}{4}
+
o_{\rm P}(1)
,
\end{equation}
where
$
\Delta^{(n)\Tb_n}
:=
{R}_n^{\rm St}/\sqrt{2}
$ is asymptotically normal with mean zero and variance~$1/2$
\linebreak (${R}_n^{\rm St}$ is the standardized Rayleigh test statistic in~(\ref{raylhdstat})). 
%
\end{Theor}

Applying Le Cam's third lemma, 
\vspace{-.8mm}
we obtain that, as~$\ny$ under~${\rm P}^{(n)\Tb_n}_{\kappa_n,f_{\rm FvML}}$, with~$\kappa_n=\tau p_n^{3/4}/\sqrt{n}$, $\Delta^{(n)\Tb_n}$ converges weakly to the normal distribution with mean~$\Gamma\tau^2$ and variance~$\Gamma$, with~$\Gamma=1/2$. The model~$\{{\rm P}^{(n)\Tb_n}_{\kappa,f_{\rm FvML}}:\kappa\geq 0\}$  (where~${\rm P}^{(n)\Tb_n}_{0,f_{\rm FvML}}:={\rm P}\n_0$) is thus ``second-order" LAN, in the sense that the mean of the limiting Gaussian shift experiment is quadratic (rather than linear) in~$\tau$. Clearly, this does not change the form of locally asymptotically optimal tests, but only their asymptotic performances. Note that the contiguity rate~$\kappa_n \sim p_n^{3/4}/\sqrt{n}$ associated with this new LAN property differs from the contiguity rate~$\kappa_n \sim \sqrt{p_n/n}$ in Theorem~\ref{LANtheor}. 

Theorem~\ref{LANinvartheor} entails that the test rejecting the null of uniformity  at asymptotic level~$\alpha$ whenever
$\Delta^{(n)\Tb_n}/\sqrt{\Gamma}=R_n^{\rm St}>z_\alpha$ (that is, the high-dimensional Rayleigh test in~(\ref{RayHDtraindelavie})) is,  in the FvML case,  \emph{locally asymptotically most powerful invariant}, that is, locally asymptotically most powerful in the class of invariant tests. This optimality result is of a high-dimensional asymptotic nature and also covers cases where~$\kappa_n$ does not converge to~$0$, hence   does not follow from the aforementioned local optimality result from \cite{Chi2003}. Le Cam's third lemma readily implies that ${R}_n^{\rm St}$ converges weakly to the normal distribution with mean~$\tau^2/\sqrt{2}$ and variance one as~$\ny$ under~${\rm P}^{(n)\Tb_n}_{\kappa_n,f_{\rm FvML}}$, with~$\kappa_n=\tau p_n^{3/4}/\sqrt{n}$, so that the corresponding asymptotic power of the Rayleigh test is given by
\begin{equation}
\label{powRay}
\lim_{n\to\infty}
{\rm P}\n_{\thetab_n,\kappa_n,f_{\rm FvML}}
\big[
R_n^{\rm St}>z_\alpha
\big]
=
1
- 
\Phi
\big( 
z_\alpha- {\textstyle\frac{\tau^2}{\sqrt{2}}}
\big)
,
\end{equation}
where the sequence~$(\thetab_n)$ is such that $\thetab_n\in\mathcal{S}^{p_n-1}$ for all~$n$ but is otherwise arbitrary. While the Rayleigh test is blind to alternatives in~$\kappa_n\sim  \sqrt{p_n/n}$, it thus detects alternatives in~$\kappa_n\sim p_n^{3/4}/\sqrt{n}$, which, in view of Theorem~\ref{LANinvartheor}, is the best that can be achieved for the \mbox{un\mbox{specified-$\thetab_n$}} problem. 

Interestingly, we might have guessed that these alternatives in~$\kappa_n\sim p_n^{3/4}/\sqrt{n}$ are those that can be detected by the high-dimensional Rayleigh test. Recall indeed that heuristic arguments in Section~\ref{oraclesec} suggested that, under ${\rm P}\n_{\thetab_n,\kappa_n,f}$, with~$\kappa_n=\tau\sqrt{p/n}$, the distribution of~$R_n^{\rm St}$ is close to $\mathcal{N}\big(\frac{\tau^2}{\sqrt{2p}}
\,
,
 1 + \frac{2\tau^2}{p}\big)
$ for large~$n$ and~$p$. Consequently, to obtain, in high dimensions, an asymptotic non-null distribution that differs from the limiting null (standard normal) one, we need to consider alternatives of the form~${\rm P}\n_{\thetab_n,\kappa_n,f}$, with~$\kappa_n=\tau p_n^{3/4}/\sqrt{n}$, under which the distribution of~$R_n^{\rm St}$ is then expected to be approximately $\mathcal{N}\big(\frac{\tau^2}{\sqrt{2}},1\big)$ for  large~$n$ and~$p$. This is fully in line with the non-null distribution and local asymptotic powers obtained from Le Cam's third lemma in the previous paragraph. 

Provided that~$f$ is four times differentiable at~$0$ and that~$p_n=o(n^2)$, tedious computations allowed to show that a 
\vspace{-.8mm}
fourth-order expansion of the $f$-based log-likelihood ratio~$\Lambda_{n,f}^{\Tb_n}$ above, still based on~$\kappa_n=\tau_n p_n^{3/4}/\sqrt{n}$, exactly provides the righthand side of~(\ref{LAQinvarmain}), with the same central sequence~$
\Delta^{(n)\Tb_n}$. However, turning this into a proper $f$-based version of Theorem~\ref{LANinvartheor} requires controlling the corresponding (fifth-order) remainder term, which proved to be extremely difficult. Yet we conjecture that Theorem~\ref{LANinvartheor} indeed extends to an arbitrary~$f$ admitting five derivatives at~$0$, under the aforementioned assumption that~$p_n=o(n^2)$ (an assumption that is superfluous in the FvML case, since Theorem~\ref{LANinvartheor} allows~$p_n$ to go to infinity in an arbitrary way as a function of~$n$). Proving this conjecture would establish that the Rayleigh test is locally asymptotically most powerful invariant under any such~$f$, with the same asymptotic powers as in~(\ref{powRay}). Since this remains a conjecture, we now study the asymptotic powers of the high-dimensional Rayleigh test away from the FvML case.


\section{Asymptotic non-null behaviour of the Rayleigh test}
\label{Rayleighsec}

In this section, we derive the asymptotic distribution of the high-dimensional Rayleigh test under rotationally symmetric  distributions that encompass those considered in Sections~\ref{contiguitysec}-\ref{invariantsec}. Here we do not require that the rotationally symmetric alternatives are monotone (in the sense of Section~\ref{contiguitysec}), nor absolutely continuous with respect to the surface area measure on the unit sphere, nor that they involve a concentration parameter~$\kappa$. Yet one of our objectives is to interpret the results of this section in the light of the contiguity/LAN/rate-consistency/power results obtained above. 


More specifically, the sequences of alternatives we consider in this section are described by triangular arrays of observations $\Xb_{ni}$, $i=1,\ldots,n$, $n=1,2,\ldots$ such that, for any~$n$, $\Xb_{n1},\Xb_{n2},\ldots,\Xb_{nn}$ are mutually independent and share a common rotationally symmetric distribution on~$\mathcal{S}^{p_n-1}$. We denote by~${\rm P}^{(n)}_{\thetab_n,F_n}$ the corresponding hypothesis when~$\Xb_{ni}$ is rotationally symmetric about~$\thetab_n$ and~$\Xb_{ni}\pr\thetab_n$ has cumulative distribution function~$F_n$. Since the Rayleigh test statistic is invariant under rotations, we will, without loss of generality, restrict to the case for which~$\thetab_n$, for any~$n$, coincides with the first vector of the canonical basis of~$\R^{p_n}$. The corresponding sequence of hypotheses will then simply be denoted as~${\rm P}^{(n)}_{F_n}$. 

Under the null of uniformity (which we still denote as~${\rm P}_0\n$), the test statistic~${R}_n^{\rm St}$ in~(\ref{raylhdstat}) has mean zero and variance~$\frac{n-1}{n}$($\to 1$). Rotationally symmetric alternatives are expected to have an impact on the asymptotic mean and variance of~${R}_n^{\rm St}$. This is made precise in the following result (see Appendix~\ref{appB1} for a proof).

\begin{Prop}
\label{propmoments}
Under~${\rm P}^{(n)}_{F_n}$, 
$
{\rm E}[{R}_n^{\rm St}]
=
(n-1)\sqrt{p_n} \, e_{n1}^2/\sqrt{2}
$
and
$\sigma_n^2
:=
p_n \tilde{e}_{n2}^2
+
2n p_n e_{n1}^2 \tilde{e}_{n2}
+
f_{n2}^2
={\rm Var}[{R}_n^{\rm St}]+o(1)$ as~$n\to\infty$, where the expectations~$e_{n\ell}:={\rm E}[(\Xb_{ni}'\thetab_n)^\ell]$, $\tilde{e}_{n\ell}:={\rm E}[(\Xb_{ni}'\thetab_n-e_{n1})^\ell]$, and $f_{n\ell}:={\rm E}[ (1- (\Xb_{ni}'\thetab_n)^2)^{\ell/2}]$ are taken under~${\rm P}^{(n)}_{F_n}$.
\end{Prop}
\vspace{2mm}

Under~${\rm P}_0\n$, $e_{n1}=0$ and $\tilde{e}_{n2}=e_{n2}=1/p_n$, so that Proposition~\ref{propmoments} is compatible with the null values of~${\rm E}[{R}_n^{\rm St}]$ and ${\rm Var}[{R}_n^{\rm St}]$ provided above. Now, parallel to the null case (see Theorem~\ref{raylnull}), the Rayleigh test statistic, after appropriate standardization, is also asymptotically standard normal under a broad class of rotationally symmetric alternatives. More precisely, we have the following result (see Appendix~\ref{appB2} for a proof).

\begin{Theor} 
\label{maintheorempower}
Let $(p_n)$ be a sequence of positive integers diverging to~$\infty$ as $\ny$. Assume that the sequence~$({\rm P}^{(n)}_{F_n})$ is such that, as $\ny$,
(i)
$\min
\big(
\frac{p_n \tilde{e}_{n2}^2}{f_{n2}^2}
,
\frac{\tilde{e}_{n2}}{ne_{n1}^2}
\big)
=o(1)$,
(ii)
$\tilde{e}_{n4}/\tilde{e}_{n2}^2=o(n)$
and 
(iii)
$f_{n4}/f_{n2}^2=o(n)$.
Then,  under~${\rm P}^{(n)}_{F_n}$,
$
(R_n^{\rm St}-{\rm E}[{R}_n^{\rm St}])/\sigma_n
= 
\frac{\sqrt{2p_n}}{n\sigma_n} 
\sum_{1\leq i< j\leq n}
\,
 \big(\Xb_{ni}\pr \Xb_{nj}-e_{n1}^2\big)
\stackrel{\mathcal{D}}{\to}
\,
\mathcal{N}(0,1)
$
as~$n\to\infty$.
\end{Theor}

This result applies under very mild assumptions, that in particular do not impose absolute continuity nor any other regularity conditions. The only structural assumptions are the conditions~(i)-(iii) above. These, however, may only be violated for rotationally symmetric distributions that are very far from the null of uniformity (hence, for alternatives under which there is in practice no need for a test of uniformity). Indeed, a necessary --- yet far from sufficient --- condition for (i)-(iii) to be violated is that~$\Xb_{n1}\pr\thetab_n$ converges in probability to some constant~$c(\in[-1,1])$. Moreover, in the FvML case, the conditions~(i)-(iii) \emph{always} hold, that is, they hold without any constraint on the concentration~$\kappa_n$ nor on the way the dimension~$p_n$ goes to infinity with~$n$ (the proof of this statement is very lengthy and requires original results on modified Bessel functions ratios, hence is provided in the supplementary article~\cite{Cut2015}). 

%


Theorems~\ref{maintheorempower} allows to compute the asymptotic power of the Rayleigh test under appropriate sequences of alternatives. As mentioned above, the null of uniformity~${\cal H}_{0n}$ yields~$e_{n1}=0$ and $\tilde{e}_{n2}=1/p_n$. Here, we therefore consider ``local" departures from uniformity of the form
$
{\cal H}_{1n}: 
\big\{ 
{\rm P}_{F_n}\n
: 
e_{n1}=0+\nu_n \tau
, 
\,
\tilde{e}_{n2}=(1/p_n) + \xi_n \eta
\big\}
\cdot 
$
The following result provides the asymptotic power of the high-dimensional Rayleigh test in~(\ref{RayHDtraindelavie}) under sequences of local alternatives that, as we will show, are intimately related to those we considered in Sections~\ref{oraclesec}-\ref{invariantsec} (see Appendix~\ref{appB2} for a proof).

\begin{Theor} 
\label{Powerprop}
Let $(p_n)$ be a sequence of positive integers diverging to~$\infty$ as $\ny$. Let the sequence~$({\rm P}^{(n)}_{F_n})$ satisfy the  assumptions of Theorem~\ref{maintheorempower} and be such that
\vspace{-1mm}
\begin{equation}
\label{lalt}
e_{n1}
=
\frac{\tau}{n^{1/2}p_n^{1/4}}
+
o\bigg(\frac{1}{n^{1/2}p_n^{1/4}}\bigg)
\qquad\textrm{and}\qquad 
\tilde{e}_{n2}
=
\frac{1}{p_n}
+
o\Big(\frac{1}{p_n}\Big)
,
\end{equation}
for some~$\tau\geq 0$. Then, under~${\rm P}\n_{F_n}$, the asymptotic power of the high-dimensional Rayleigh test in~(\ref{RayHDtraindelavie}) is given by
$
1
- 
\Phi
\big( 
z_\alpha- \frac{\tau^2}{\sqrt{2}}
\big).
$
\end{Theor}


In order to link these alternatives to those considered earlier, note that, as~$\ny$ under~${\rm P}\n_{\thetab_n,\kappa_n,f}$, with $\kappa_n=\xi_n \sqrt{p_n/n}$, where the positive sequence~$(\xi_n)$ is~$o(\sqrt{n})$, we have
\begin{eqnarray}
e_{n1}
&\!\!=\!\!&
\Big(\frac{c_{p_n}}{c_{p_n,\kappa_n,f}}\Big)^{-1}
\,
\frac{c_{p_n}}{\kappa_n}
\int_{-1}^1 (1-s^2)^{(p_n-3)/2} \, \kappa_n s f(\kappa_n s)\,ds
\nonumber
\\[2mm]
&\!\!=\!\!&
\bigg(
1+
\frac{\kappa_n^2}{2p_n}f''(0)
+
o\bigg(\frac{\kappa^2_n}{p_n} \bigg)
\bigg)^{-1}
\bigg(
\frac{\kappa_n}{p_n}   + 
o\bigg(\frac{\kappa_n}{p_n} \bigg)
\bigg)
\label{genexpande1}
\end{eqnarray}
and
\begin{eqnarray}
e_{n2}
\nonumber
&\!\!=\!\!&
\Big(\frac{c_{p_n}}{c_{p_n,\kappa_n,f}}\Big)^{-1}
\,
\frac{c_{p_n}}{\kappa_n^2}
\int_{-1}^1 (1-s^2)^{(p_n-3)/2} \, (\kappa_n s)^2 f(\kappa_n s)\,ds
\\[2mm]
&\!\!=\!\!&
\bigg(
1+
\frac{\kappa_n^2}{2p_n}f''(0)
+
o\bigg(\frac{\kappa^2_n}{p_n} \bigg)
\bigg)^{-1}
\bigg(
\frac{1}{p_n}   + 
o\bigg(\frac{1}{p_n} \bigg)
\bigg)
\label{genexpande2}
,
\end{eqnarray}
where we used four times Lemma~\ref{lemcontig}.
For the contiguous alternatives in Theorem~\ref{LANtheor}, that is for~${\rm P}\n_{\thetab_n,\kappa_n,f}$, with $\kappa_n=\tau_n \sqrt{p_n/n}$ (where~$(\tau_n)$ is bounded),
(\ref{genexpande1})-(\ref{genexpande2}) provide
\begin{equation}
\label{fsh}
e_{n1}
=
\frac{\tau_n}{\sqrt{np_n}}
+
o\bigg(\frac{1}{\sqrt{np_n}}\bigg)
\qquad\textrm{and}\qquad 
\tilde{e}_{n2}
=
\frac{1}{p_n}
+
o\bigg(\frac{1}{p_n}\bigg)
.
\end{equation}
Theorem~\ref{Powerprop} implies that the asymptotic power of the  high-dimensional Rayleigh test under the alternatives~(\ref{fsh}) is equal to~$\alpha$, which confirms (see Section~\ref{oraclesec}) that this test is blind to contiguous alternatives. 

Now, at least if~$p_n=o(n^2)$ (a constraint that is actually superfluous in the FvML case, as it can be seen by using the Amos-type bounds provided in Lemma~S.3.2 from~\cite{Cut2015}), the more
\vspace{-1mm}
 severe alternatives~${\rm P}\n_{\thetab_n,\kappa_n,f}$, with $\kappa_n=\tau p_n^{3/4}/\sqrt{n}$, from Theorem~\ref{LANinvartheor} translate --- still in view of~(\ref{genexpande1})-(\ref{genexpande2}) --- into those in~(\ref{lalt}). This shows that the asymptotic powers of the high-dimensional Rayleigh test computed in the FvML case via Le Cam's third lemma (see~(\ref{powRay})) actually also hold away from the FvML case. Clearly, this further supports the conjecture from Section~\ref{invarianthighsec} that, under the assumption that~$p_n=o(n^2)$, Theorem~\ref{LANinvartheor} holds for an essentially arbitrary~$f$.

\section{A Monte Carlo study}
\label{simusec}

In this section, we present the results of a Monte Carlo study we conducted to check the validity of our asymptotic results. We performed two simulations. In the first one, we generated independent random samples of the form
\begin{equation} 
\label{samples}
\Xb_{i;j}^{(\ell)} 
\quad 
i=1, \ldots, n, 
\quad
j=1,2, 
\quad
\ell=0,1,2,3,4 
.
 \end{equation}
For~$\ell=0$, the common distribution of the $\Xb_{i;j}^{(\ell)}$'s is the uniform distribution on the unit sphere~$\mathcal{S}^{p-1}$, while, for~$\ell>0$, the $\Xb_{i;j}^{(\ell)}$'s have an FvML distribution on~$\mathcal{S}^{p-1}$ with location $\thetab=(1, 0, \ldots, 0)\pr \in \R^p$ and concentration~$\kappa_j^{(\ell)}$, with
$$
\kappa_1^{(\ell)}=0.6\ell \, \sqrt{\frac{p}{n}}
\quad
\textrm{ and }
\quad
\kappa_2^{(\ell)}=0.6\ell\, \frac{p^{3/4}}{\sqrt{n}}
\cdot
$$
In the second simulation, we considered again independent random samples of the form~(\ref{samples}), still with $\Xb_{i;j}^{(0)}$'s that are uniform over~$\mathcal{S}^{p-1}$. Here, however, the $\Xb_{i;j}^{(\ell)}$'s, for~$\ell=1,2,3,4$, are rotationally symmetric with location $\thetab=(1, 0, \ldots, 0)\pr \in \R^p$ and are such that the $\thetab\pr \Xb_{i;j}^{(\ell)}$'s are beta 
with mean~$e_{1;j}^{(\ell)}$ and variance~$\tilde{e}_{2;j}=1/p$, where we let 
$$
e_{1;1}^{(\ell)}
=
\frac{0.6\ell}{\sqrt{np}}
\quad
\textrm{ and }
\quad
e_{1;2}^{(\ell)}
=
\frac{0.6\ell}{n^{1/2}p^{1/4}}
$$ 
(this beta example is associated with a non-monotonic nuisance~$f$, which is allowed in Section~\ref{Rayleighsec}). In both simulations, the value $\ell=0$ corresponds to the null hypothesis of uniformity, while $\ell=1,2,3,4$ provide increasingly severe alternatives. The case~$j=1$ relates to the contiguous alternatives (see Theorem~\ref{LANtheor}) and the corresponding (more general) alternatives in~(\ref{fsh}), whereas~$j=2$ is associated with the alternatives under which the Rayleigh test shows non-trivial asymptotic powers in the high-dimensional setup (see Theorem~\ref{LANinvartheor} and the alternatives~(\ref{lalt})).  

For any~$(n,p)\in C\times C$, with $C:=\{30, 100, 400\}$, any~$j\in\{1,2\}$, and any~$\ell\in\{0,1,2,3,4\}$, we generated $M=2,500$ independent random samples~$\Xb_{i;j}^{(\ell)}$, $i=1, \ldots, n$, as described above, and evaluated the rejection frequencies of~(i) the \mbox{specified-$\thetab_n$} test~$\phi_{\thetab_n}\n$ in~(\ref{deforayoup}) and of (ii) the high-dimensional Rayleigh test~$\phi\n$ in~(\ref{RayHDtraindelavie}), both conducted at nominal level~$5\%$. Rejection frequencies are plotted in Figures~\ref{FigFvML} and~\ref{Figbeta}, for FvML and beta-type alternatives, respectively. In each figure, we also plot the corresponding asymptotic powers, obtained from~(\ref{poworacle}), (\ref{powRay}), Theorem~\ref{Powerprop}, and the fact that~$\phi_{\thetab_n}\n$ 
is consistent against ($j=2$)-alternatives. 

Clearly, for both simulations, rejection frequencies match extremely well the corresponding asymptotic powers, irrespective of the tests and types of alternatives considered (the only possible exception is the test~$\phi_{\thetab_n}\n$ under ($\ell=1,j=2$)-alternatives; this, however, is only a consequence of the lack of continuity of the corresponding asymptotic power curves). Remarkably, this agreement is also reasonably good for moderate sample size~$n$ and dimension~$p$. Beyond validating our asymptotic results, this Monte Carlo study therefore also shows that these results are relevant for practical values of~$n$ and~$p$.

\begin{figure}[htbp!]  
\begin{center}
\includegraphics[width=1.03\linewidth]{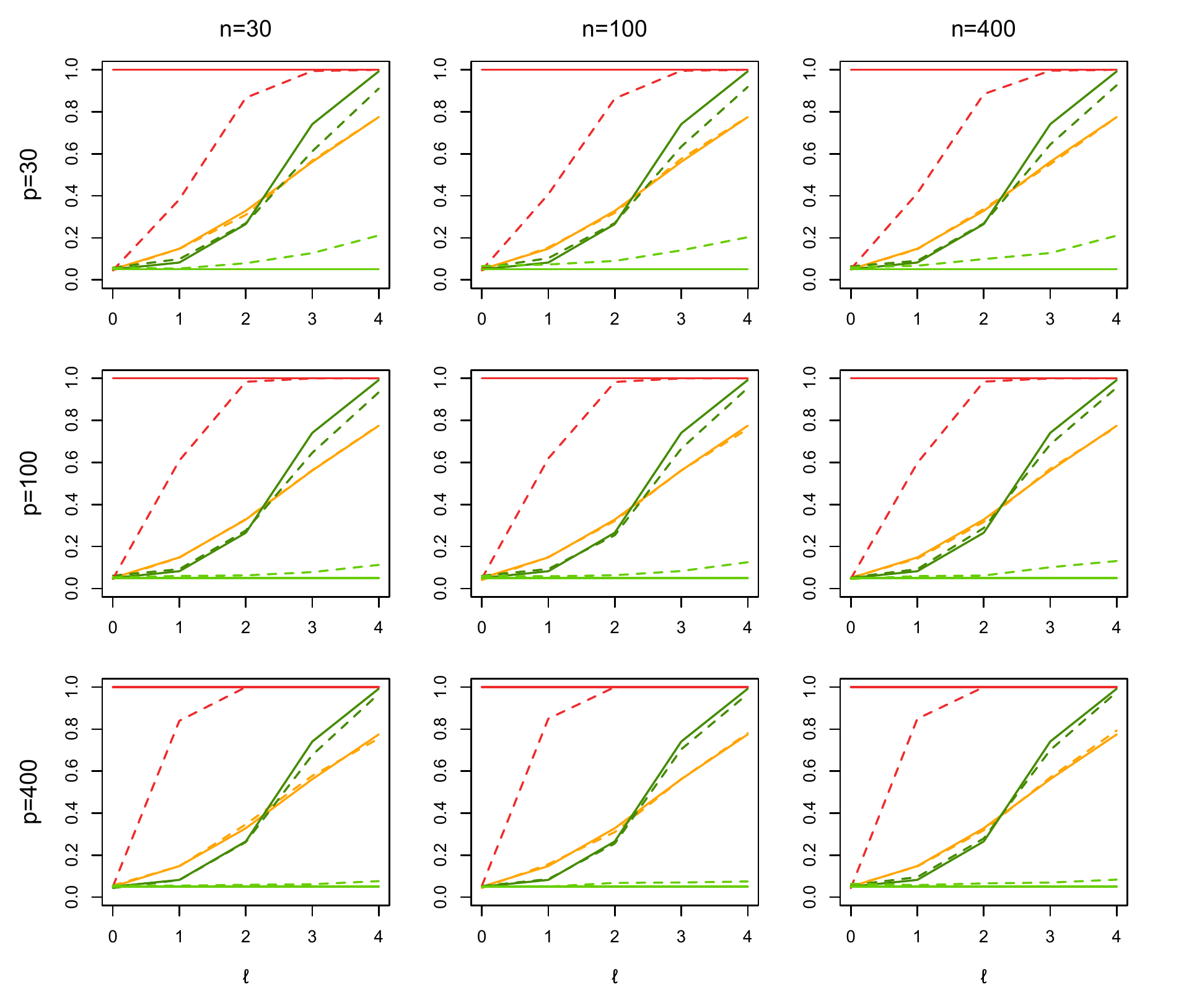}
\caption{\small Rejection frequencies (dashed) and asymptotic powers (solid), under the null of uniformity over the $p$-dimensional unit sphere ($\ell=0$) and increasingly severe FvML alternatives ($\ell=1,2,3,4$), of the \mbox{specified-$\thetab_n$} test~$\phi_{\thetab_n}\n$  in~(\ref{deforayoup}) (red/orange) and the high-dimensional Rayleigh test~$\phi\n$ in~(\ref{RayHDtraindelavie}) (light/dark green). Light colors (orange and light green) are associated with contiguous alternatives, whereas dark colors (red and dark green) correspond to the more severe alternatives under which the Rayleigh test shows non-trivial asymptotic powers  in high dimensions; see Section~\ref{simusec} for details.}
\label{FigFvML}
\end{center}
\end{figure}

\begin{figure}[htbp!] 
\begin{center}
\includegraphics[width=1.03\linewidth]{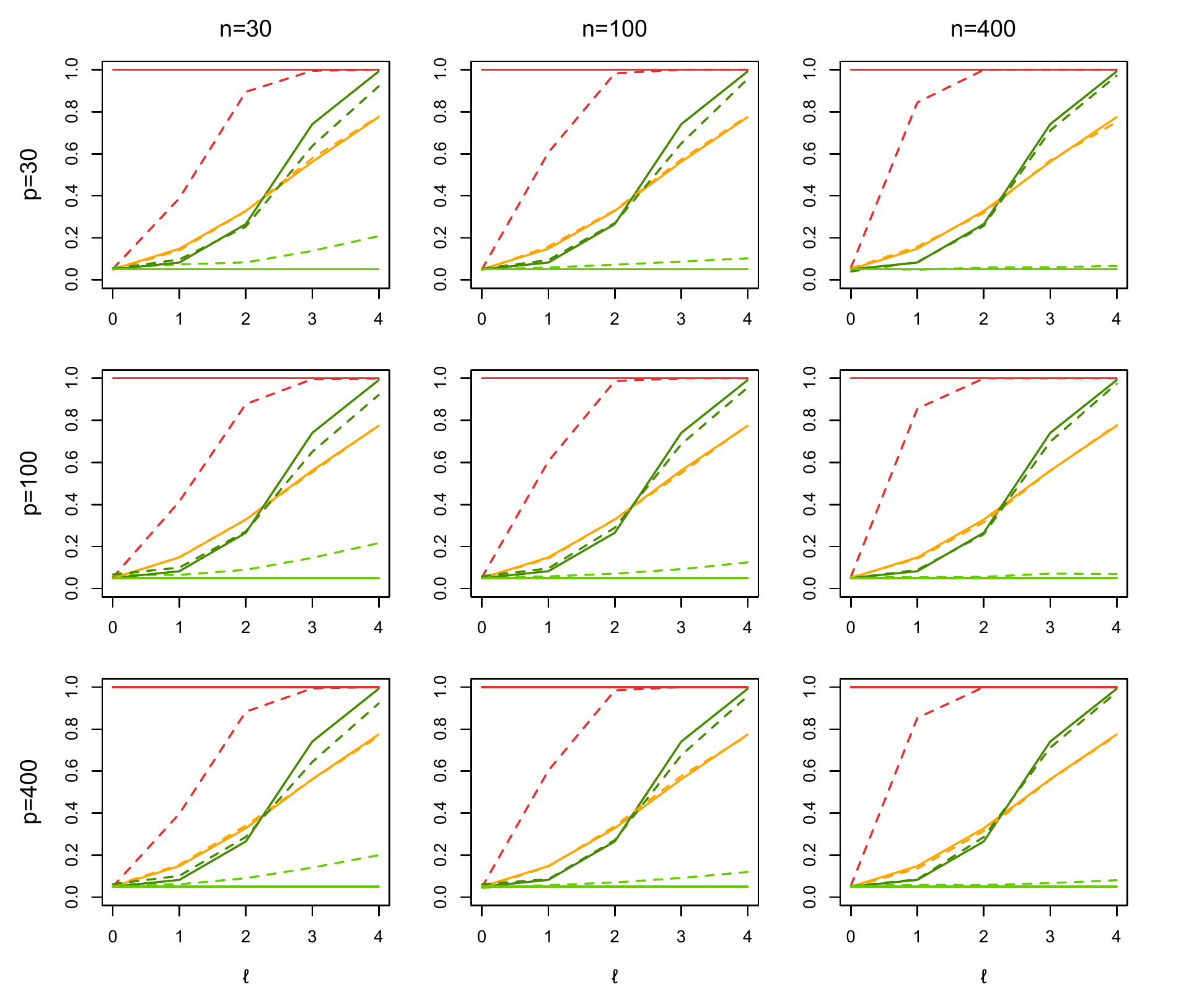}
\caption{\small Rejection frequencies (dashed) and asymptotic powers (solid), under the null of uniformity over the $p$-dimensional unit sphere ($\ell=0$) and  increasingly severe ``beta" rotationally symmetric  alternatives ($\ell=1,2,3,4$), of the \mbox{specified-$\thetab_n$} test~$\phi_{\thetab_n}\n$ in~(\ref{deforayoup}) (red/orange) and the high-dimensional Rayleigh test~$\phi\n$ in~(\ref{RayHDtraindelavie}) (light/dark green). Light colors (orange and light green) are associated with contiguous alternatives, whereas dark colors (red and dark green) correspond to the more severe alternatives under which the Rayleigh test shows non-trivial asymptotic powers  in high dimensions; see Section~\ref{simusec} for details.}
\label{Figbeta}
\end{center} 
\end{figure}


\section{An application}
\label{realsec}

Since the seminal paper \cite{LedWol2002}, one of the most widely considered testing problems in high-dimensional statistics is the problem of testing for sphericity. A possible approach to test for sphericity about a specified centre (without loss of generality, the origin of~$\R^p$) is to perform a test of uniformity on the sphere~$\mathcal{S}^{p-1}$ on ``spatial signs", that is, on the observations projected on~$\mathcal{S}^{p-1}$; see, among others, \cite{Caietal2013}, where this is used in a possibly high-dimensional setup, and \cite{cueetal2009}, where it is argued that ``\emph{in most practical cases the violations of sphericity will arise from the non-fulfillment of uniformity on the unit sphere for projected data}". This is particularly true in the high-dimensional case, since the concentration-of-measure phenomenon there implies that information lie much more in the directions of the observations from the origin than in their distances from the origin (incidentally, note that  \cite{JuaPri2001} also invoked the same argument to adopt a directional approach for outlier detection in high dimensions). 

As showed in the previous sections, the high-dimensional Rayleigh test 
will show power against \emph{skewed} rotationally symmetric distributions on the sphere (skewness arises from the monotonicity of the corresponding nuisance~$f$). 
On the contrary, the Rayleigh test will be blind to any non-spherical distribution in~$\R^p$ whose projection on the sphere charges antipodal regions equally. In particular, it will show no power against elliptical alternatives, hence also against spiked alternatives (that is, against alternatives associated with scatter matrices of the form $\Sigb=\sigma ({\bf I}_p+ \lambda \betab \betab\pr)$, with $\sigma,\lambda>0$ and $\betab\in\mathcal{S}^{p-1}$). Interestingly, most (if not all) tests for sphericity in high dimensions are designed to detect elliptical or spiked alternatives. This is the case, e.g., both for the Gaussian sphericity test ($\phi_{\mathcal{N}}\n$, say) from \cite{Joh1972} and for the sign test of sphericity ($\phi_S\n$, say) from~\cite{HP06} (these tests were shown to be valid in high-dimensions in \cite{LedWol2002} and~\cite{Zouetal2013}/\cite{PaiVer2015}, respectively). 
In line with this, theoretical efforts have so far focused on spiked alternatives; see, e.g., \cite{Onaetal2013,Onaetal2014}, where powers of various tests of sphericity under  high-dimensional spiked alternatives were investigated. 

To illustrate these antagonistic power behaviours, we performed the following simulation exercise involving the Gaussian sphericity test~$\phi_{\mathcal{N}}\n$ and the sign sphericity test~$\phi_S\n$ (both in their version to test for sphericity about the origin of~$\R^p$), as well as the high-dimensional Rayleigh test~$\phi_R\n$. 
%
%
%
%
For~$\ell=0,1,2,3,4$ and~$n=p=100$, we generated 10,000 $p$-dimensional independent samples~$\Xb_{i; \ell}^{(1)}$, $i=1, \ldots, n$, and~$\Xb_{i; \ell}^{(2)}$, $i=1, \ldots, n$, from two different alternatives to sphericity~: 
\begin{itemize}
\item[(i)] $\Xb_{i; \ell}^{(1)}$, $i=1, \ldots, n$ form a random sample from the $p$-variate skew-normal distribution with location vector~${\bf 0}$, scatter matrix~${\bf I}_p$ and skewness vector~$(\ell, \ldots, \ell)'\in\R^p$; see \cite{AzzCap1999};
\item[(ii)] $\Xb_{i; \ell}^{(2)}$, $i=1, \ldots, n$ form a random sample from the $p$-variate normal distribution with mean~${\bf 0}$ and covariance matrix~${\bf I}_p+ \ell {\bf e}_1{\bf e}_1\pr$, with~${\bf e}_1=(1,0,\ldots,0)'\in\R^p$.
\end{itemize}  
For both~(i)-(ii), $\ell=0$ is associated with the null of sphericity about the origin of~$\R^p$, whereas $\ell=1,2,3,4$ provide increasingly severe alternatives. Figure \ref{powerss} plots the resulting empirical powers of the three tests mentioned above, all performed at nominal level~$5\%$. Results confirm that the Rayleigh test~$\phi_R\n$ performs quite well under alternatives of type~(i) but shows no power against alternatives of type~(ii), whereas the tests~$\phi_{\mathcal{N}}\n$ and~$\phi_S\n$ do the exact opposite. In practice, thus, as soon as the Rayleigh test and more standard tests of sphericity lead to opposite rejection decisions, practitioners are offered some insight on what type of deviation from sphericity they are likely to be facing.
 

\begin{figure}[!htbp]
\includegraphics[width=\linewidth, height=73mm]{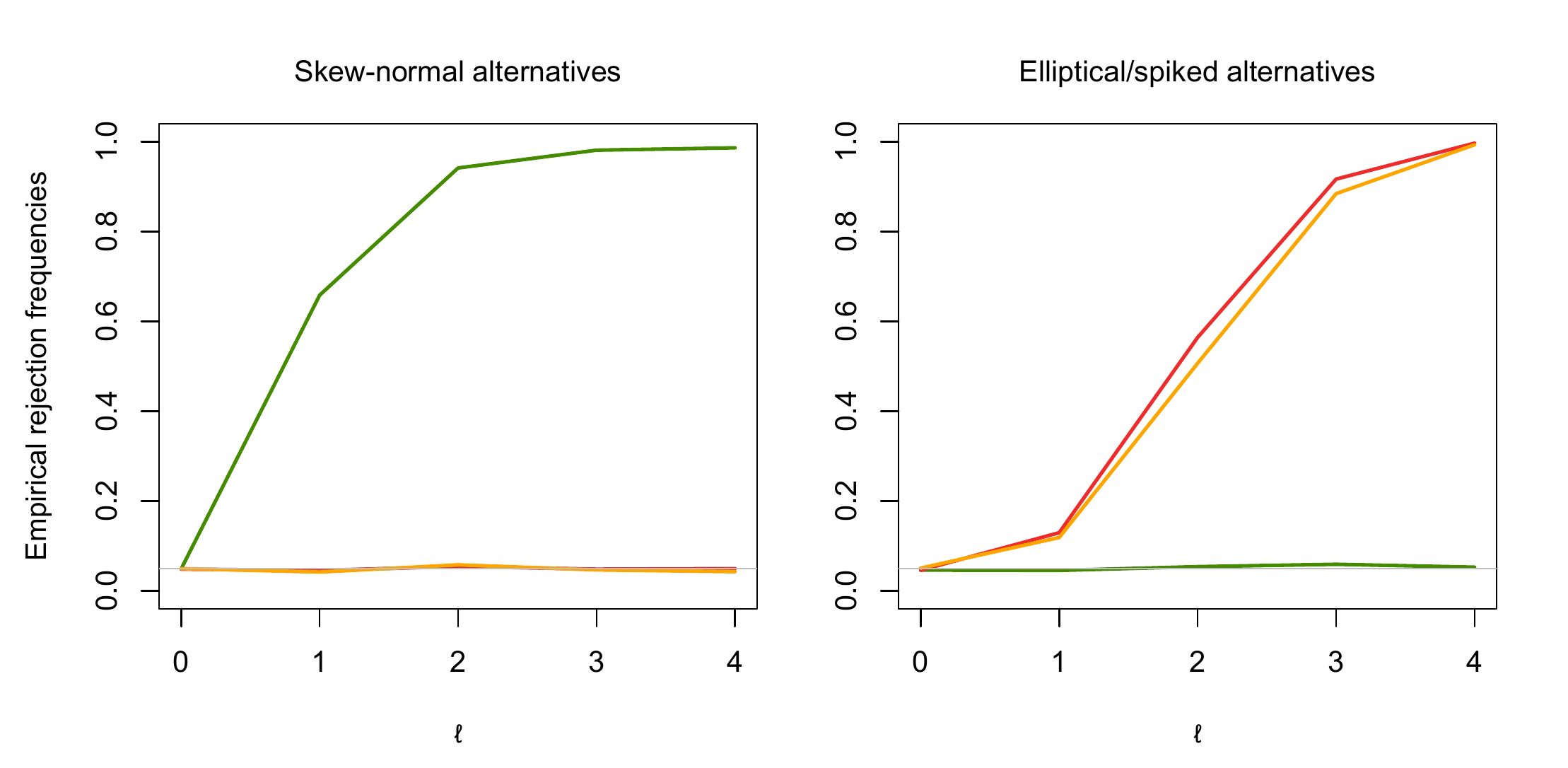} 
\caption{ \small (Left:) Rejection frequencies, under the null of sphericity in~$\R^p$ ($\ell=0$) and increasingly severe skew-normal 
\vspace{-.5mm}
alternatives ($\ell=1,2,3,4$), of the Rayleigh sphericity test~$\phi_R\n$ (green), along with the more classical
 sphericity tests~$\phi_{S}\n$ (orange) and~$\phi_{\mathcal{N}}\n$ (red).
(Right:) The corresponding rejection frequencies under some $p$-variate spiked alternatives. In both cases, the dimension~$p$ and the sample size~$n$ are equal to~$100$, the nominal level is~$5\%$, and the number of replications is~$10,000$; 
see Section~\ref{realsec} for details.
}
 \label{powerss}    
\end{figure} 

We illustrate this on a real data example. We considered the gene expression dataset analyzed in \cite{Eisetal1998}; more precisely, we restricted to a subsample of~100 ribosomal proteins from this dataset. The data then take the form of a matrix ${\bf X}=(X_{ij})$, where~$X_{ij}$ is the $j$th expression value ($j=1, \ldots, p=79$) of the $i$th gene ($i=1,\ldots,n=100$) (even though~$n>p$, the present data may be considered high-dimensional since the small value of~$n/p$ prevents relying on fixed-$p$ asymptotic results). The rows of~$\Xb$ (the ``expression vectors") are obtained from DNA microarray experiments. After imputing missing data (by replacing any missing entry in~$\Xb$ with the sample average of available measurements on the same variable) and centering the observations via the sample average, we performed the Rayleigh test~$\phi_R\n$ and its competitors~$\phi_{\mathcal{N}}\n$ and~$\phi_S\n$. While centering still leaves some space for rejection by~$\phi_R\n$ (that is indeed based on the sample average of \emph{projected} observations), the Rayleigh test interestingly provides a $p$-value above~$.9999$,
whereas those of~$\phi_{\mathcal{N}}\n$ and~$\phi_S\n$ are below~$.0001$.
Hence, at any usual nominal level, the null of sphericity is rejected, and the outcome of the various tests suggest that the deviation from sphericity is of an elliptical, or at least of a centrally symmetric, nature.
This may be useful to guide further modelling of this gene expression dataset.

\section{Conclusions and perspectives}
\label{conclusec}

In this final section, we summarize the results of the paper and present perspectives for future research.

\subsection{Summary}

We considered the problem of testing uniformity on the unit sphere in the low- and high-dimensional setups. Rotationally symmetric alternatives with modal location~$\thetab_n$, concentration~$\kappa_n$ and functional parameter~$f$ were considered. We showed that~$\kappa_n\sim \sqrt{p_n/n}$ provides contiguous alternatives. For specified~$\thetab_n$, a local asymptotic normality result was established (at the aforementioned contiguity rate), which allowed, both in low and high dimensions, to define locally asymptotically most powerful tests for the \mbox{specified-$\thetab_n$} problem.

In practice, however, $\thetab_n$ may rarely be assumed to be known. In the corresponding \mbox{un\mbox{specified-$\thetab_n$}} problem, we showed that the Rayleigh test enjoys nice asymptotic optimality properties, both in the low- and high-dimensional cases. In low dimensions, it is locally asymptotically maximin, irrespective of~$f$. In high dimensions, it is locally asymptotically most powerful invariant in the FvML case, and a conjecture --- that is strongly supported by a fourth-order expansion of the relevant $f$-based local log-likelihood ratio and by the computation of asymptotic powers in Section~\ref{Rayleighsec} --- states that, provided that~$p_n=o(n^2)$, this optimality holds for any~$f$ that is five times differentiable at~$0$. 

Our results fully characterize the cost of the possible unspecification of~$\thetab_n$. In low dimensions, this cost is in terms of asymptotic powers but not in terms of rate. In high-dimensions, however, there is a cost in terms of rate, as optimal tests cannot detect  the contiguous alternatives in~$\kappa_n\sim \sqrt{p_n/n}$, but only the more severe alternatives in~$\kappa_n\sim p_n^{3/4}/\sqrt{n}$. Simulation results are in remarkable agreement with our asymptotic results, irrespective of the relative magnitude of~$n$ and~$p$ --- which materializes the robustness of most of our results in the rate at which~$p_n$ goes to infinity with~$n$. A real data example illustrated the usefulness of the high-dimensional Rayleigh test in the framework of testing for sphericity.

\subsection{Perspective for future research}

In the distributional framework described in Section~\ref{contiguitysec}, the problem of testing uniformity consists in testing the null hypothesis that the concentration parameter~$\kappa_n$ is equal to zero. Depending on the information at hand, the other parameters, namely the modal location~$\thetab_n$ and the infinite-dimensional parameter~$f$, may be regarded as specified or unspecified. If~$f$ is specified, then the problem is of a parametric nature and optimality quite naturally relates to the local asymptotic normality of the corresponding fixed-$f$ submodel (both the specified- and unspecified-$\thetab_n$ parametric problems can be considered). We showed that, for any sufficiently smooth~$f$ in a neighbourhood of the origin, the test in~(\ref{deforayoup}) and the Rayleigh test achieve the $f$-parametric efficiency bounds in the specified- and unspecified-$\thetab_n$ problems, respectively.

Since it can hardly be assumed in practice that~$f$ is known, it is more natural to adopt a semiparametric point of view under which~$f$ remains unspecified. The optimality results stated in this paper should then be read in a semiparametric sense, under unspecified~$f$ in the specified-$\thetab_n$ problem, and under unspecified~$(\thetab_n,f)$ in the unspecified-$\thetab_n$ one. In all cases, such results are pointwise in~($\thetab_n,f$) and relate to the corresponding semiparametric efficiency bounds at~($\thetab_n,f$); see, e.g., \cite{Bic1998}. In the present setup, it is not needed to go through tangent space calculations to derive the resulting semiparametrically optimal tests; indeed, since the test in~(\ref{deforayoup}) and the Rayleigh test are parametrically optimal at any smooth~$f$ (for the specified- and unspecified-$\thetab_n$ problems, respectively), they also are semiparametrically optimal at such~$f$. Another corollary of our results is that semiparametrically efficiency bounds at~($\thetab_n,f$)  do not depend on~$(\thetab_n,f)$ but differ in the specified- and unspecified-$\thetab_n$ problems.

Now, the problem of testing uniformity over the unit sphere is primarily of a \emph{nonparametric} nature. Even if the distributional framework described in Section~\ref{contiguitysec} is considered, it is therefore valid to adopt a nonparametric point of view and to try to identify, e.g., \emph{minimax separation rates}; see, e.g., \cite{Ing2000} or, in a directional context,  \cite{Fayetal2013}, \cite{TM14} and \cite{TM15}. This approach is fundamentally different from the semiparametric one adopted in this work. In particular, instead of providing \emph{pointwise} results in~$(\thetab_n,f)$, this approach aims at identifying consistency rates that are associated with the \emph{worst-in-}$f$ (resp., \emph{worst-in-}($\thetab_n,f)$) performances that can be achieved in the specified-$\thetab_n$ (resp., \mbox{unspecified-$\thetab_n$}) problem. This fundamental difference between the semiparametric and nonparametric approaches above does not make it possible to translate our results in terms of minimax separation rates. Nevertheless, preliminary results indicate that, at least for the specified-$\thetab_n$ problem, the consistency rates described in this paper are also minimax separation rates and that the test~$\phi_{\thetab_n}\n$ is ``rate-optimal in the minimax sense" (obviously, it would be natural to consider further the unspecified-$\thetab_n$ problem).  These results, however, require much work and rely on other techniques than those considered in the present paper, hence will be presented elsewhere.

\appendix

\section{Proofs for Sections~2 to~4}
\label{appA}

The proofs of Theorems~\ref{contigtheor} and~\ref{LANtheor} require the following preliminary result. 


\begin{Lem}
\label{lemcontig}
Let $g:\R\to \R$ be twice differentiable at~$0$. Let $\kappa_n$ be a positive sequence that is $o(\sqrt{p_n})$ as~$\ny$. Then
$
R_n(g)
:=
c_{p_n}
\int_{-1}^1 (1-s^2)^{(p_n-3)/2} g(\kappa_n s)  \,ds
= 
g(0)
+
\frac{\kappa_n^2}{2p_n} g''(0)
+
o\big( \frac{\kappa_n^2}{p_n} \big)
.
$
\end{Lem}

{\sc Proof of Lemma~\ref{lemcontig}.} 
First note that, if~$\Xb$ is uniformly distributed over~$\mathcal{S}^{p_n-1}$ (hence is such that~${\rm Var}[\Xb]=(1/p_n)\mathbf{I}_{p_n}$), then~(\ref{unifcdf}) implies that
\begin{equation}
\label{foref1}
c_{p_n}
\int_{-1}^1 s^2 (1-s^2)^{(p_n-3)/2}  \,ds
=
{\rm E}[(\thetab'\Xb)^2]
=
\thetab' {\rm E}[\Xb\Xb']\thetab
=
\frac{1}{p_n}
,
\end{equation}
where~$\thetab\in\mathcal{S}^{p_n-1}$ is arbitrary. Now, write
$$
R_n(g)
=
g(0)
+
c_{p_n}
\int_{-1}^1 (1-s^2)^{(p_n-3)/2} (g(\kappa_n s)-g(0)-\kappa_n s g'(0))  \,ds
.
$$
Letting~$t=\kappa_n s$ and using the identity~(\ref{foref1}) with~$p=p_n$ then provides
$$
R_n(g)
=
g(0)+
\frac{\kappa_n^2}{p_n}
\int_{-\kappa_n}^{\kappa_n} h_n(t) \bigg( \frac{g(t)-g(0)-t g'(0)}{t^2} \bigg)
\,dt
,
$$
where~$h_n$ is defined through 
$$
t
\mapsto
h_n(t)
=
\frac{(\textstyle\frac{t}{\kappa_n})^2(1-(\textstyle\frac{t}{\kappa_n})^2)^{(p_n-3)/2}}{\int_{-\kappa_n}^{\kappa_n}(\textstyle\frac{s}{\kappa_n})^2 (1-(\textstyle\frac{s}{\kappa_n})^2)^{(p_n-3)/2}\,ds}
\,
\mathbb{I}[|t|\leq \kappa_n]
.
$$
It can be checked that, since $\kappa_n=o(\sqrt{p_n})$, the~$h_n$'s form an \emph{approximate $\delta$-sequence}, in the sense that
$
\int_{-\infty}^\infty h_n(t)\,dt=1 
$
for all~$n$ and
$
\int_{-\varepsilon}^\varepsilon h_n(t)\, dt \to 1
$
for any~$\varepsilon>0$. Hence, 
$$
R_n(g)
=
g(0)
+
\frac{\kappa_n^2}{p_n}
\,
\lim_{t\to 0} \bigg( \frac{g(t)-g(0)-t g'(0)}{t^2} \bigg)
+
o\Big( \frac{\kappa_n^2}{p_n} \Big)
$$
which, by using L'H\^{o}pital's rule, yields the result. 
\cqfd
\vspace{3mm}

{\sc Proof of Theorem~\ref{contigtheor}}.
In this proof, all expectations and variances are taken under the null of uniformity~${\rm P}_{0}\n$ and all stochastic convergences and $o_{\rm P}$'s are as~$\ny$ under~${\rm P}_{0}\n$.
Consider then the local log-likelihood ratio   
\begin{eqnarray}
\lefteqn{
\Lambda_n
:=
\log \frac{d{\rm P}\n_{\thetab_n,\kappa_n,f}}{d{\rm P}\n_{0}} 
=
\sum_{i=1}^n 
\,
\log  \frac{c_{p_n,\kappa_n,f} f(\kappa_n \Xb_{ni}\pr \thetab_n)}{c_{p_n}} 
\nonumber
}
\\
& &
\hspace{10mm}
= 
n 
\Big(
\! \log \frac{c_{p_n,\kappa_n,f}}{c_{p_n}}+E_{n1}
\Big)
+
\sum_{i=1}^n 
\big(
\log f(\kappa_n \Xb_{ni}\pr \thetab_n) - E_{n1}
\big)
=:
L_{n1}+L_{n2};
\label{toexpand}
\end{eqnarray}
throughout, 
we write $\ell_{f,k}:=(\log f)^k$ and~$E_{nk}:={\rm E}\big[ \ell_{f,k}(\kappa_n \Xb_{ni}\pr\thetab_n) \big]$ ($E_{nk}$ actually depends on~$\kappa_n$, $p_n$ and~$f$, but we simply write $E_{nk}$ to avoid a heavy notation). 

Lemma~\ref{lemcontig} readily yields
\begin{eqnarray}
\hspace{-4mm}
\nonumber
n \log \frac{c_{p_n,\kappa_n,f}}{c_{p_n}}
&=&
- n \log 
\bigg( 
c_{p_n}
\int_{-1}^1 (1-s^2)^{(p_n-3)/2} f(\kappa_n s)\,ds 
\bigg)
\\[2mm]
&= & 
-n\log 
\bigg(
1+
\frac{\kappa_n^2}{2p_n}
f''(0) 
+
o\Big(\frac{\kappa^2_n}{p_n} \Big)
\bigg)
=
-\frac{n\kappa_n^2}{2p_n} f''(0) 
+
o\Big(\frac{n\kappa^2_n}{p_n} \Big)
.
\label{expandlogconst}
\end{eqnarray}
Similarly, for any positive integer~$k$,
\begin{equation}
E_{nk}
=
c_{p_n}
\int_{-1}^1 (1-s^2)^{(p_n-3)/2}  \ell_{f,k}(\kappa_n s) \,ds 
=
\frac{\kappa^2_n}{2p_n} \ell_{f,k}''(0)
+
o\Big(\frac{\kappa^2_n}{p_n}\Big)
.
\label{expandEk}
\end{equation}
Combining~(\ref{expandlogconst}) and~(\ref{expandEk}), and using the identity $\ell_{f,1}''(0)
=f''(0)-1$ readily yields
$$
L_{n1}
=
\frac{n\kappa^2_n}{2p_n} 
\Big(
-f''(0)
+
 \ell_{f,1}''(0)
\Big)
+
o\Big(\frac{n\kappa^2_n}{p_n} \Big)
=
-
\frac{n\kappa^2_n}{2p_n} 
+
o\Big(\frac{n\kappa^2_n}{p_n} \Big)
.
$$

Turning to~$L_{n2}$, write
$$
L_{n2}
=
\sqrt{n V_n}
\,
\sum_{i=1}^n 
W_{ni}
:=
\sqrt{n V_n}
\,
\sum_{i=1}^n 
\frac{\log f(\kappa_n \Xb_{ni}\pr \thetab_n) - E_{n1}}{\sqrt{n V_n}}
,
$$
where we let
$
V_n
:=
{\rm Var}\big[ \log f(\kappa_n \Xb_{ni}\pr\thetab_n) \big]
.
$
First note that (\ref{expandEk}) provides 
\begin{equation}
\label{expandV}
nV_n
=
n\big(E_{n2} - E_{n1}^2\big)
=
\frac{n\kappa^2_n}{2p_n} \ell_{f,2}''(0)
+
o\Big(\frac{n\kappa^2_n}{p_n} \Big)
=
\frac{n\kappa^2_n}{p_n} 
+
o\Big(\frac{n\kappa^2_n}{p_n} \Big)
,
\end{equation}
which leads to
\begin{equation}
\label{quasicont}
\Lambda_n
=
-
\frac{n\kappa^2_n}{2p_n} 
+
\sqrt{\frac{n\kappa^2_n}{p_n} 
+
o\Big(\frac{n\kappa^2_n}{p_n} \Big)}
\sum_{i=1}^n
W_{ni}
+
o\Big(\frac{n\kappa^2_n}{p_n} \Big)
.
\end{equation}
Since $W_{ni}$, $i=1,\ldots,n$ are mutually independent with mean zero and variance~$1/n$, we obtain that
\begin{equation}
\label{L2norm}
{\rm E}
\big[
\Lambda_n^2
\big]
=
\big(
{\rm E}
\big[
\Lambda_n
\big]
\big)^2
+
{\rm Var}
\big[
\Lambda_n
\big]
=
\frac{n^2\kappa^4_n}{4p_n^2} 
+
o\Big(\frac{n^2\kappa^4_n}{p_n^2} \Big)
+
\frac{n\kappa^2_n}{p_n} 
+
o\Big(\frac{n\kappa^2_n}{p_n} \Big)
.
\end{equation}
If $\kappa^2_n=o(\frac{p_n}{n})$, then~(\ref{L2norm}) implies that $\exp(\Lambda_n)\to 1$, so that Le Cam's first lemma yields that ${\rm P}\n_{\thetab_n,\kappa_n,f}$ and ${\rm P}\n_{0}$ are mutually contiguous. 

We may therefore assume that $\kappa^2_n=\tau^2_n p_n/n$, where the positive sequence~($\tau_n$) is~$O(1)$ but not~$o(1)$. In this case, (\ref{quasicont}) rewrites
$
\Lambda_n
=
-
\frac{\tau^2_n}{2} 
+
\sqrt{\tau_n^2 
+
o(1)}
\sum_{i=1}^n
W_{ni}
+
o(1)
.
$

Applying the Cauchy-Schwarz inequality and the Chebychev inequality, then using~(\ref{expandEk}) and~(\ref{expandV}), provides that, for some positive constant~$C$,
\begin{eqnarray*}
\lefteqn{
\hspace{-10mm}
\sum_{i=1}^n 
{\rm E}[
W_{ni}^2 
\mathbb{I}[ |W_{ni}| >\varepsilon]]
\leq 
n
\sqrt{{\rm E}[ W_{ni}^4] {\rm P}[|W_{ni}|>\varepsilon]}
\leq
\frac{n}{\varepsilon}
\sqrt{{\rm E}[ W_{ni}^4] {\rm Var}[W_{ni}]}
=
\frac{\sqrt{n}}{\varepsilon}
\sqrt{{\rm E}[ W_{ni}^4]}
}
\\[2mm]
& & 
\hspace{5mm}
\leq
\frac{Cn^{1/2}E_{n4}^{1/2}}{\varepsilon (nV_n)^2}
=
\frac{C
\Big(
\frac{n\kappa^2_n \ell_{f,4}''(0)}{2p_n} 
+
o\big(\frac{n\kappa^2_n}{p_n} \big)
\Big)^{1/2}
}{
\varepsilon
\Big(
\frac{n\kappa^2_n}{p_n} 
+
o\big(\frac{n\kappa^2_n}{p_n} \big)
\Big)^2
}
=
\frac{
o(\tau_n)
}{
\varepsilon
\big(
\tau_n^2 
+
o(\tau_n^2)
\big)^2
}
=o(1)
,
\end{eqnarray*}
where we have used the fact that $\ell_{f,4}''(0)=0$. This shows that $\sum_{i=1}^n W_{ni}$ satisfies the classical Levy-Lindeberg condition,  hence is asymptotically standard normal (as already mentioned, $W_{ni}$, $i=1,\ldots,n$ are mutually independent with mean zero and variance~$1/n$). For any subsequence~$(\exp(\Lambda_{n_m}))$ converging in distribution, we must then have
$
\exp(\Lambda_{n_m}) \to \exp(Y)
$,
with
$
Y
\sim
\mathcal{N}( 
-\frac{1}{2} \lim_{n\to\infty} \tau_{n_m}^2
,
 \lim_{n\to\infty} \tau_{n_m}^2 
)
.
$
Mutual contiguity ${\rm P}\n_{\thetab_n,\kappa_n,f}$ and ${\rm P}\n_{0}$ then follows from the fact that ${\rm P}[\exp(Y)=0]=0$ and ${\rm E}[\exp(Y)]=1$.
\cqfd
\vspace{3mm}

{\sc Proof of Theorem~\ref{LANtheor}.} 
As in the proof of Theorem~\ref{contigtheor}, all expectations and variances in this proof are taken under the null of uniformity~${\rm P}_{0}\n$ and all stochastic convergences and $o_{\rm P}$'s are as~$\ny$ under~${\rm P}_{0}\n$. The central limit theorem then directly establishes Part~(ii) of the result, since 
${\rm E}[\Delta_{\thetab_n}\n]=0$ and 
$
{\rm Var}[\textstyle\Delta_{\thetab_n}\n]
=
\frac{p_n}{n} 
{\rm Var}\big[\sum_{i=1}^n  \Xb_{ni}\pr \thetab_n\big]
=
1.
$

It therefore remains to establish Part~(i). Recall that, in the case where  $(\tau_n)$ is~$O(1)$ but not~$o(1)$, we have obtained in the proof of Theorem~\ref{contigtheor} that 
$$
\Lambda_n
=
-
\frac{\tau^2_n}{2} 
+
\sqrt{\tau_n^2 
+
o(1)}
\,
\sum_{i=1}^n
W_{ni}
+
o(1)
=
-
\frac{\tau^2_n}{2} 
+
\tau_n 
\,
\sum_{i=1}^n
W_{ni}
+
o_{\rm P}(1)
,
$$
where 
$
\sum_{i=1}^n W_{ni}
=
(1/\sqrt{nV_n})
\sum_{i=1}^n
(\log f(\kappa_n \Xb_{ni}\pr \thetab_n) - E_{n1})
$
is asymptotically standard normal. To establish the result, it is therefore sufficient to show that~$\tau_n[(\sum_{i=1}^n W_{ni})-\Delta_{\thetab_n}\n]$ converges to zero in quadratic mean. To do so, write
$$
\tau_n\Big(\sum_{i=1}^n W_{ni}\Big)
-
\tau_n\Delta_{\thetab_n}\n
=
\frac{\tau_n}{\sqrt{nV_n}}
\sum_{i=1}^n
\big(
\log f(\kappa_n \Xb_{ni}\pr \thetab_n) - E_{n1}
-
\sqrt{p_nV_n}
\,
\Xb_{ni}\pr \thetab_n
\big)
=:
\frac{M_n}{\sqrt{nV_n}} 
\cdot
$$
Then using ${\rm E}[\Xb_{n1}\pr \thetab_n]=0$ and ${\rm E}[(\Xb_{n1}\pr \thetab_n)^2]=1/p_n$, we obtain
\begin{eqnarray*}
{\rm E}
\big[
M_n^2
\big]
&=&
n\tau_n^2
\,
{\rm E}
\big[
\big(
\log f(\kappa_n \Xb_{ni}\pr \thetab_n) - E_{n1}
- \sqrt{p_nV_n}\,\Xb_{ni}\pr \thetab_n
\big)^2
\big]
\\[2mm]
&=&
n\tau_n^2
\,
(
2V_n
-2\sqrt{p_nV_n}
 \,
{\rm E}[
\Xb_{ni}\pr \thetab_n
(
\log f(\kappa_n \Xb_{ni}\pr \thetab_n) - E_{n1}
)
]
)
\\[2mm]
&=&
2n\tau_n^2   V_n 
-2\tau_nn^{3/2} \sqrt{V_n}\,
{\rm E}[
\kappa_n\Xb_{ni}\pr \thetab_n
\log f(\kappa_n \Xb_{ni}\pr \thetab_n) 
]
,
\end{eqnarray*}
which, letting $g(x):=x(\log f(x))$, provides 
\begin{equation}
\label{hgd}
{\rm E}
\Big[
\Big(
\tau_n\Big(\sum_{i=1}^n W_{ni}\Big)
-
\tau_n\Delta_{\thetab_n}\n
\Big)^2
\Big]
=
2\tau_n^2 
-\frac{2\tau_n n}{\sqrt{nV_n}} \,
{\rm E}[
g(\kappa_n \Xb_{ni}\pr \thetab_n) 
]
.
\end{equation}
Using Lemma~\ref{lemcontig},
$$
{\rm E}[
g(\kappa_n \Xb_{ni}\pr \thetab_n) 
]
=
c_{p_n}
\int_{-1}^1 (1-s^2)^{(p_n-3)/2}  g(\kappa_n s) \,ds 
=
\frac{\kappa^2_n}{2p_n} g''(0)
+
o\Big(\frac{\kappa^2_n}{p_n}\Big)
=
\frac{\kappa^2_n}{p_n} 
+
o\Big(\frac{\kappa^2_n}{p_n}\Big)
.
$$
Plugging in~(\ref{hgd}) and using~(\ref{expandV}) then yields
$$
{\rm E}
\Big[
\Big(
\tau_n\Big(\sum_{i=1}^n W_{ni}\Big)
-
\tau_n \Delta_{\thetab_n}\n
\Big)^2
\Big]
=
2\tau_n^2 
-\frac{2\tau_n 
\Big(
\frac{n\kappa^2_n}{p_n} 
+
o\Big(\frac{n\kappa^2_n}{p_n}\Big)
\Big)
}{\Big(
\frac{n\kappa^2_n}{p_n} 
+
o\Big(\frac{n\kappa^2_n}{p_n} \Big)
\Big)
^{1/2}} 
=
o(1)
,
$$
as was to be showed. 
\cqfd
\vspace{2mm}

{\sc Proof of Theorem~\ref{LANinvartheor}}. 
The FvML version of the log-likelihood in~(\ref{invarL}) rewrites
\begin{equation}
\label{gus}
\Lambda_{n,f_{\rm FvML}}^{\Tb_n}
=
n \,
\! \log \frac{c_{p_n,\kappa_n,f_{\rm FvML}}}{c_{p_n}}
+
\log 
{\rm E}[\exp(\kappa_n n \bar{\Xb}_{n}\pr \Ub) | \bar{\Xb}_{n}]
=:
L_{n1}+L_{n2}
,
\end{equation}
where 
\vspace{-1mm}
\begin{eqnarray*}
L_{n1}
=
-n \,
\! \log 
\frac{\Gamma\big(\frac{p_n}{2}\big)\mathcal{I}_{\frac{p_n}{2}-1}(\kappa_n)}{(\kappa_n/2)^{\frac{p_n}{2}-1}}
=:
-n \,
\! \log 
H_{\frac{p_n}{2}-1}(\kappa_n)
\end{eqnarray*}
(see~(\ref{unifcdf})-(\ref{fvmlcdf}) for explicit expressions of~$c_p$ and~$c_{p,\kappa,f_{\rm FvML}}=c^{\rm FvML}_{p,\kappa}$, respectively) 
and 
\begin{eqnarray*}
\lefteqn{
\hspace{-14mm}
L_{n2}
=
\log {\rm E}\Big[\exp\Big(n \kappa_n \|\bar{\Xb}_{n}\| \frac{\Ub\pr\bar{\Xb}_{n}}{\|\bar{\Xb}_{n}\|}\Big) \big| \bar{\Xb}_{n}\Big]
=
\log \bigg( 
c_{p_n}
\int_{-1}^1 (1-s^2)^{(p_n-3)/2} \exp(n \kappa_n  \|\bar{\Xb}_{n}\| s)\,ds 
\bigg)
}
\\[2mm]
& &
=
\log \bigg( 
\frac{c_{p_n}}{c_{p_n,n \kappa_n  \|\bar{\Xb}_{n}\|}^{\rm FvML}}
\bigg)
=
\log
H_{\frac{p_n}{2}-1}(n \kappa_n  \|\bar{\Xb}_{n}\|)
=:
\log
H_{\frac{p_n}{2}-1}( \kappa_n \sqrt{n}T_n)
.
\end{eqnarray*}

Now, using the bounds
$
S_{\nu+\frac{1}{2},\nu+\frac{3}{2}}(\kappa) \leq \log H_\nu(\kappa)\leq S_{\nu,\nu+2}(\kappa)
$
(see~(5) in \citealp{HotGru2014})
with
$
S_{\alpha,\beta}(\kappa)
:=
\sqrt{\kappa^2+\beta^2}-\alpha \log(\alpha+\sqrt{\kappa^2+\beta^2})-\beta+\alpha \log(\alpha+\beta)
,
$
one easily obtains that, if~$\kappa_n=\tau_n p_n^{3/4}/\sqrt{n}$, with~$n,p_n\to\infty$ and~$(\tau_n)$ bounded,
\begin{equation}
\label{LInv1}
L_{n1} + \frac{n\kappa_n^2}{2p_n} - \frac{n\kappa_n^4}{4p_n^2(p_n+2)}
=o(1)
\ \
\textrm{ and }
\ \
L_{n2} - \frac{n\kappa_n^2}{2p_n}T_n^2 + \frac{n^2\kappa_n^4}{4p_n^2(p_n+2)}T_n^4
=o_{\rm P}(1)
\end{equation}
under~${\rm P}\n_0$. The first (resp., second) result in~(\ref{LInv1}) requires expanding the log term in the $S_{\alpha,\beta}$ bounds as
$\log x=x-1-\frac{1}{2c}(x-1)^2$ (resp., $\log x=x-1-\frac{1}{2}(x-1)^2+\frac{1}{3c^2}(x-1)^3$) for some~$c$ between~$x$ and~1, and the second one further requires using the fact that~$T_n=1+o_{\rm P}(1)$  as~$\ny$ under~${\rm P}_0\n$ (which directly follows from Theorem~\ref{raylnull}). Plugging~(\ref{LInv1}) into~(\ref{gus}) and using again the fact that~$T_n=1+o_{\rm P}(1)$ entails that, as~$\ny$ under~${\rm P}_0\n$, 
$$
\Lambda_{n,f_{\rm FvML}}^{\Tb_n}
=
-
\frac{n\kappa_n^2}{2p_n}
+
\frac{n \kappa_n^2}{2p_n}\, T_n^2
-
\frac{n^2 \kappa_n^4}{4p_n^2(p_n+2)}\, T_n^4
+o_{\rm P}(1)
=
\tau_n^2\,\frac{R_n^{\rm St}}{\sqrt{2}} 
-
\frac{\tau^4_np_n}{4(p_n+2)}
+o_{\rm P}(1)
.
$$
Jointly with Theorem~\ref{raylnull}, this establishes the result.
\cqfd
\vspace{3mm}


\section{Proofs for Section~5}

In this second appendix, we establish Proposition~\ref{propmoments} and Theorems~\ref{maintheorempower}-\ref{Powerprop}. 

\subsection{Preliminary lemmas and proof of Proposition~\ref{propmoments}}
\label{appB1}

Define the quantities
$u_{ni}:=\Xb_{ni}'\thetab_n$ and $v_{ni}:=(1- u_{ni}^2)^{1/2}$
that are associated with the \emph{tangent-normal decomposition} 
$\Xb_{ni}=u_{ni}\thetab_n+v_{ni} {\bf  S}_{ni}$ of~$\Xb_{ni}$, where $\Sb_{ni}:=
(\Xb_{ni}-(\Xb_{ni}'\thetab_n)\thetab_n)/\|\Xb_{ni}-(\Xb_{ni}'\thetab_n)\thetab_n\|$ if $\Xb_{ni}\neq \thetab_n$ and~$\mathbf{0}$
otherwise. With this notation, $e_{n\ell}={\rm E}[u_{ni}^\ell]$ and $f_{n\ell}={\rm E}[v_{ni}^\ell]$ (see Proposition~\ref{propmoments}). We start with the following lemmas.

\begin{Lem} 
\label{ref1} 
Under ${\rm P}\n_{F_n}$,
(i) ${\rm E}[{\bf  S}_{ni}{\bf  S}_{ni}']=\frac{1}{p_n-1} (\mathbf{I}_{p_n}-\thetab_n\thetab_n')$ for any~$i$;
(ii) ${\rm E}[({\bf  S}_{ni}'{\bf  S}_{nj})^2]=\frac{1}{p_n-1}$ for any~$i\neq j$;
(iii) ${\rm E}[({\bf  S}_{ni}'{\bf  S}_{nj})^4]=\frac{3}{p^2_n-1}$ for any~$i\neq j$.
\end{Lem}

\begin{proof}
(i) Let~$\mathbf{O}$ be a~$p_n\times p_n$ orthogonal matrix such that~$\mathbf{O}\thetab_n=\mathbf{e}_1$, where~$\mathbf{e}_1$ denotes the first vector of the canonical basis of~$\R^{p_n}$. Then the random vectors~$\mathbf{O}{\bf  S}_{ni}$, $i=1,\ldots,n$ form a random sample from the uniform distribution over~$\{ \xb\in\mathcal{S}^{p_n-1} : \mathbf{e}_1' \xb=0  \}$. Consequently, 
$
\mathbf{O}{\rm E}[{\bf  S}_{ni} {\bf  S}_{ni}']\mathbf{O}'
=
\frac{1}{p_n-1} (\mathbf{I}_{p_n}-\mathbf{e}_1\mathbf{e}_1')
,
$
which yields the result. 
\linebreak
(ii)-(iii) 
It follows from the joint distribution of the $\mathbf{O}{\bf  S}_{ni}$'s just derived that, for any~$i\neq j$, ${\bf  S}_{ni}'{\bf  S}_{nj}=(\mathbf{O}{\bf  S}_{ni})'(\mathbf{O}{\bf  S}_{nj})$ is equal in distribution to~$\Ub'\Vb$, where the independent random $(p_n-1)$-vectors $\Ub$, $\Vb$ are uniformly distributed over~$\mathcal{S}^{p_n-2}$. The result then follows from Lemma~A.1(iii) in \cite{PaiVer2015}.   
\end{proof}

\begin{Lem} 
\label{moments} 
Under ${\rm P}\n_{F_n}$, 
(i) ${\rm E}[\Xb_{ni}\pr\Xb_{nj}]=e_{n1}^2$ for any~$i\neq  j$,
(ii) ${\rm E}[(\Xb_{ni}\pr\Xb_{nj})^2]=e_{n2}^2+f_{n2}^2/(p_n-1)$ for any~$i\neq  j$,
(iii) ${\rm E}[(\Xb_{ni}\pr\Xb_{nk})(\Xb_{n\ell}\pr\Xb_{nj})]=e_{n2}e_{n1}^2$ for any~$i\neq j$ and~$k\neq \ell$ such that~$\{i,j,k,\ell\}$ contains exactly three different indices, and
(iv) ${\rm E}[(\Xb_{ni}\pr\Xb_{nj})(\Xb_{nk}\pr\Xb_{n\ell})]=e_{n1}^4$ if~$i,j,k,\ell$ are pairwise different.
\end{Lem}

\begin{proof}
Part~1 of the lemma directly follows from
$
\Xb_{ni}\pr\Xb_{nj}
=
(u_{ni}\thetab_n+v_{ni} {\bf  S}_{ni})'(u_{nj}\thetab_n+v_{nj} {\bf  S}_{nj})
=
u_{ni}u_{nj}+v_{ni}v_{nj} {\bf  S}_{ni}'{\bf  S}_{nj}.
$
For the remaining claims, 
note that, for~$i<j$ and $k<\ell$, 
\begin{equation}
\label{to}
{\rm E}[(\Xb_{ni}\pr\Xb_{nj})(\Xb_{nk}\pr\Xb_{n\ell})]
=
{\rm E}[u_{ni}u_{nj} u_{nk}u_{n\ell}]
+
{\rm E}[v_{ni}v_{nj} v_{nk}v_{n\ell}]
{\rm E}[({\bf  S}_{ni}'{\bf  S}_{nj}) ( {\bf  S}_{nk}'{\bf  S}_{n\ell})]
.
\end{equation}
Part~2 of the result then follows from Lemma~\ref{ref1}(ii). For Parts 3-4, there is always one of the indices~$i,j,k,\ell$ that is different from the other three indices, which implies that ${\rm E}[({\bf  S}_{ni}'{\bf  S}_{nj}) ( {\bf  S}_{nk}'{\bf  S}_{n\ell})]=0$. The result readily follows.
\end{proof}

Lemma~\ref{moments} allows to prove Proposition~\ref{propmoments}. 
\vspace{2mm}

{\sc Proof of Proposition~\ref{propmoments}}. 
Since the expectation readily follows from Lemma~\ref{moments}(i), we can focus on the variance. Using Lemma~\ref{moments}(i) again, we obtain  
$$
{\rm Var}_{F_n}[{R}_n^{\rm St}]
=
\frac{2p_n}{n^2} \sum_{1\leq i< j\leq n}\sum_{1\leq k<\ell\leq n} \Big( {\rm E}[(\Xb_{ni}\pr\Xb_{nj})(\Xb_{nk}\pr\Xb_{n\ell})] - e_{n1}^4 \Big)
.
$$
In this sum, there are ${n \choose 2}$ terms corresponding to Lemma~\ref{moments}(ii) and $6{n \choose 4}$ terms (not contributing to the sum) corresponding to Lemma~\ref{moments}(iv). Therefore, there are 
$
{n \choose 2}^2-{n \choose 2}-6{n \choose 4}=n(n-1)(n-2)
$
terms corresponding to Lemma~\ref{moments}(iii). Consequently, 
\begin{eqnarray*}
{\rm Var}_{F_n}[R_{n}^{\rm St}]
&=&
\frac{2p_n}{n^2} 
\bigg\{
{n \choose 2}
 \Big( e_{n2}^2+ f_{n2}^2/(p_n-1) - e_{n1}^4 \Big)
+
n(n-1)(n-2)
 \Big( e_{n2}e_{n1}^2 - e_{n1}^4 \Big)
\bigg\}
\\[2mm]
&=&
\frac{p_n(n-1)}{n} 
\bigg\{
 \big( e_{n2}^2 - e_{n1}^4 \big) 
+
2(n-2)
e_{n1}^2 \big( e_{n2} - e_{n1}^2 \big)
+ f_{n2}^2/(p_n-1)
\bigg\}
,
\end{eqnarray*}
which, since $\tilde{e}_{n2}=e_{n2} - e_{n1}^2$, establishes the result.
\cqfd
\vspace{3mm}

Both following lemmas are needed to prove Theorem~\ref{maintheorempower}. 
\vspace{0mm}

\begin{Lem} 
\label{moments2} 
Under ${\rm P}\n_{F_n}$, 
(i)
 $
{\rm E}
\big[
\big(
\Xb_{ni}-e_{n1}\thetab_n 
\big)
\big(
\Xb_{ni}-e_{n1}\thetab_n 
\big)\pr
\big]
=
\tilde{e}_{n2} \thetab_n \thetab_n\pr
+
\frac{f_{n2}}{p_n-1} (\mathbf{I}_{p_n}-\thetab_n \thetab_n\pr)
;
$
(ii) 
$
{\rm Var}
\big[
\big(
\Xb_{ni}\pr\thetab_n-e_{n1} 
\big)
\big(
\Xb_{nj}\pr\thetab_n-e_{n1} 
\big)
\big]
=
\tilde{e}_{n4}-\tilde{e}_{n2}^2$
for~$i=j$ and 
$\tilde{e}_{n2}^2$ for $i\neq j$;
(iii)
$
{\rm E}
\big[
\Xb_{ni}\pr
(\mathbf{I}_{p_n}-\thetab_n \thetab_n\pr)
\Xb_{nj}
\big]
=
f_{n2}$
for $i=j$ and 0  for $i\neq j$;
(iv)
$
{\rm Var}
\big[
\Xb_{ni}\pr
(\mathbf{I}_{p_n}-\thetab_n \thetab_n\pr)
\Xb_{nj}
\big]
=
f_{n4}-f_{n2}^2$
for $i=j$ and $f_{n2}^2/(p_n-1)$ for $i\neq j$. 
\end{Lem}

\begin{proof}
(i) Using the tangent-normal decomposition and  Lemma~\ref{ref1}(i), we obtain
\begin{eqnarray*}
\lefteqn{
{\rm E}
\big[
(
\Xb_{ni}-e_{n1}\thetab_n 
)
(
\Xb_{ni}-e_{n1}\thetab_n 
)\pr
\big]
=
{\rm E}
\big[
(
(u_{ni}-e_{n1})\thetab_n +v_{ni} {\bf  S}_{ni}
)
(
(u_{ni}-e_{n1})\thetab_n +v_{ni} {\bf  S}_{ni}
)\pr
\big]
}
\\[2mm]
& & 
\hspace{10mm}
=
{\rm E}
\big[
(u_{ni}-e_{n1})^2
\big]
\thetab_n \thetab_n \pr
+
f_{n2} 
\,
{\rm E}
\big[
{\bf  S}_{ni}{\bf  S}_{ni}\pr
\big]
=
\tilde{e}_{n2}
\thetab_n \thetab_n \pr
+
\frac{f_{n2}}{p_n-1} (\mathbf{I}_{p_n}-\thetab_n \thetab_n\pr)
.
\end{eqnarray*}
(ii)-(iv) The results readily follow from the fact that 
$\Xb_{ni}\pr\thetab_n-e_{n1}=u_{ni}-e_{n1}$ and
$\Xb_{ni}\pr
(\mathbf{I}_{p_n}-\thetab_n \thetab_n\pr)
\Xb_{nj}=v_{ni}v_{nj} \Sb_{ni}\pr \Sb_{nj}$ (and from  Lemma~\ref{ref1}(ii)).
\end{proof}

\begin{Lem} 
\label{moments4} 
Consider expectations of the form
$
c_{ijrs}
=
{\rm E}
\left[
\Delta_{i\ell}\Delta_{j\ell}\Delta_{r\ell}\Delta_{s\ell}
\right]
$
taken under~${\rm P}\n_{F_n}$, with~$\Delta_{i\ell}:=(\Xb_{ni}-e_{n1}\thetab_n)'(\Xb_{n\ell}-e_{n1}\thetab_n)$ and~$i\leq j\leq r\leq s<\ell$. Then
(i) 
$c_{ijrs}=
\tilde{e}_{n4}^2
+
\frac{6}{p_n-1}\big({\rm E}\big[v_{ni}^2(u_{ni}-e_{n1})^2\big]\big)^2
+
\frac{3f_{n4}^2}{p_n^2-1}
$ if $i=j=r=s$;
(ii)
$c_{ijrs}=
\tilde{e}_{n2}^2\tilde{e}_{n4} 
+
\frac{2\tilde{e}_{n2}f_{n2}}{p_n-1}{\rm E}\big[v_{ni}^2(u_{ni}-e_{n1})^2\big]
+
\frac{f_{n2}^2f_{n4}}{(p_n-1)^2}
$ if $i=j<r=s$;
(iii)
$c_{ijrs}=0$ otherwise.
\end{Lem}

\begin{proof}
We start with the proof of~(iii). Assume that $j=r$, so that we are not in case~(ii). Since case~(i) is excluded, we have $i<j$ or $r<s$. In both cases, one of the four indices~$i,j,r,s$ is different from the other three indices. Since ${\rm E}[\Delta_{i\ell}]=0$, we obtain that~$c_{ijrs}=0$, which establishes~(iii). 
Turning to the proof of (i)-(ii), we use the tangent-normal decomposition again to write~$\Delta_{j\ell}$ as~$
(u_{nj}-e_{n1})(u_{n\ell}-e_{n1})+v_{nj} v_{n\ell} ({\bf  S}_{nj}'{\bf  S}_{n\ell})$. Since ${\rm E}[({\bf  S}_{nj}'{\bf  S}_{n\ell})^k]=0$ for any odd integer~$k$, this leads to decomposing~$c_{jjrr}$ into
\begin{eqnarray*}
\lefteqn{
c_{jjrr}
=
{\rm E}
\left[
(u_{nj}-e_{n1})^2(u_{nr}-e_{n1})^2(u_{n\ell}-e_{n1})^4
\right]
}
\\[2mm]
& & 
\hspace{1mm}
+
4
{\rm E}
\left[
(u_{nj}-e_{n1})
(u_{nr}-e_{n1})
(u_{n\ell}-e_{n1})^2 
v_{nj}v_{nr}
v_{n\ell}^2
(\Sb_{nj}'\Sb_{n\ell})
(\Sb_{nr}'\Sb_{n\ell})
\right]
\\[2mm]
& & 
\hspace{1mm}
+
2
{\rm E}
\left[
(u_{nr}-e_{n1})^2(u_{n\ell}-e_{n1})^2
v_{nj}^2v_{n\ell}^2(\Sb_{nj}'\Sb_{n\ell})^2
\right]
+
{\rm E}
\left[
v_{nj}^2v_{nr}^2v_{n\ell}^4
(\Sb_{nj}'\Sb_{n\ell})^2 (\Sb_{nr}'\Sb_{n\ell})^2
\right]
.
\end{eqnarray*}
The result then follows from Lemma~\ref{ref1}(ii)-(iii).
\end{proof}


\subsection{Proofs of Theorems~\ref{maintheorempower} and~\ref{Powerprop}}
\label{appB2}

The proof of Theorem~\ref{maintheorempower} is based on the following central limit theorem for martingale differences.

\begin{Theor}[Billingsley 1995, Theorem~35.12]
\label{Billingsley}
\hspace{-3.5mm} Let~$D_{n\ell}$, $\ell=1,\ldots,n$, $n=1,2,\ldots,$ be a triangular array of random variables such that, for any~$n$, $D_{n1},D_{n2},\ldots,D_{nn}$ is a martingale difference sequence with respect to some filtration~$\mathcal{F}_{n1},\mathcal{F}_{n2},\ldots,\mathcal{F}_{nn}$. Assume that, for any~$n,\ell$, $D_{n\ell}$ 
has a finite variance. Letting~$\sigma^2_{n\ell}={\rm E}\big[D_{n\ell}^2\,|\, \mathcal{F}_{n,\ell-1}\big]$ (with~$\mathcal{F}_{n0}$ being the trivial $\sigma$-algebra~$\{\emptyset,\Omega\}$ for all~$n$), further assume that, as~$n\to\infty$, 
\begin{equation}
\label{Condition1}
\sum_{\ell=1}^n \sigma^2_{n\ell}=1+o_{\rm P}(1)
\quad 
\textrm{ and }
\quad 
\sum_{\ell=1}^n {\rm E}\big[D_{n\ell}^2\,\mathbb{I}[|D_{n\ell}|>\varepsilon]\big]\to 0
.
\end{equation}
Then $\sum_{\ell=1}^n D_{n\ell}$ is asymptotically standard normal. 
\end{Theor}

Writing~${\rm E}_{n\ell}$ for the conditional expectation with respect to the $\sigma$-algebra~${\cal F}_{n\ell}$ generated by~$\Xb_{n1}, \ldots, \Xb_{n\ell}$, we have
$$
{\rm E}_{n\ell}\big[ {R}_n^{\rm St} \big]
=
\frac{\sqrt{2p_n}}{n \sigma_n}  
\,
\bigg\{
\sum_{1\leq i< j\leq \ell}
\big(
\Xb_{ni}\pr \Xb_{nj} -e_{n1}^2
\big)
+
(n-\ell)
e_{n1}
\sum_{i=1}^\ell
\big(
\Xb_{ni}\pr \thetab_n -e_{n1}
\big)
\bigg\}
\cdot
$$
Note that
$
R_n^{\rm St}
=
\sum_{\ell=1}^n D_{n\ell}
$, where $D_{n\ell}
:= {\rm E}_{n\ell}\big[R_n^{\rm St}\big]-{\rm E}_{n,\ell-1}\big[R_n^{\rm St}\big]
$ rewrites
\begin{equation}
\label{Dnell}
D_{n\ell}
=
\frac{\sqrt{2p_n}}{n \sigma_n}  
\,
\bigg\{
\sum_{i=1}^{\ell-1}
(\Xb_{ni}-e_{n1}\thetab_n)
+
(n-1)
e_{n1}
\thetab_n
\bigg\}\pr
\big(\Xb_{n\ell}-e_{n1} \thetab_n\big)
,
\ \ 
\ell=1,2,\ldots
\end{equation}
(throughout, sums over empty set of indices are defined as being equal to zero). The following lemmas then establish the conditions~(\ref{Condition1}) in the present context.

\begin{Lem} \label{Lemmeoneone}
Let the assumptions of Theorem~\ref{maintheorempower} hold. Then, under~${\rm P}^{(n)}_{F_n}$, 
(i) $\sum_{\ell=1}^n {\rm E} [\sigma^2_{n\ell}]$ converges to one as $\ny$, and 
(ii) ${\rm Var} [\sum_{\ell=1}^n \sigma^2_{n\ell}]$ converges to zero as $\ny$.
\end{Lem}

\begin{Lem}
\label{THElemmaspher}
Let the assumptions of Theorem~\ref{maintheorempower} hold and fix~$\varepsilon>0$. Then, under~${\rm P}^{(n)}_{F_n}$,  $\sum_{\ell=1}^n {\rm E}\big[(D_{n\ell})^2 \, {\mathbb I}[| D_{n\ell}| > \varepsilon]\big]\!\to 0$ as~$n\to \infty$.  
\end{Lem}
\vspace{-1mm}

In the rest of the paper, $C$ is a positive constant that may change from line to line. 
\vspace{1mm} 

{\sc Proof of Lemma~\ref{Lemmeoneone}}. 
(i) Note that
\begin{eqnarray*}
\sigma^2_{n\ell}
&=&
\frac{2p_n}{n^2\sigma^2_n}  
\bigg\{ 
\sum_{i=1}^{\ell-1} 
(\Xb_{ni}-e_{n1}\thetab_n)
+
(n-1) e_{n1} \thetab_n
\bigg\}\pr
{\rm E}
\Big[
\big(
\Xb_{n\ell}-e_{n1}\thetab_n 
\big)
\big(
\Xb_{n\ell}-e_{n1}\thetab_n 
\big)\pr
\Big]
\\[1mm]
& &
\hspace{15mm}
\times \,
\bigg\{ 
\sum_{j=1}^{\ell-1} 
(\Xb_{nj}-e_{n1}\thetab_n)
+
(n-1) e_{n1} \thetab_n
\bigg\}.
\end{eqnarray*}
By using Lemma~\ref{moments2}(i), we obtain
\begin{eqnarray}
\sigma^2_{n\ell}
&\!\!=\!\!&
\frac{2p_n \tilde{e}_{n2}}{n^2\sigma^2_n}  
\Bigg\{ 
\sum_{i,j=1}^{\ell-1} 
(\Xb_{ni}\pr\thetab_n-e_{n1})
(\Xb_{nj}\pr\thetab_n-e_{n1})
+
2(n-1) e_{n1}
\sum_{i=1}^{\ell-1} 
(\Xb_{ni}\pr\thetab_n-e_{n1})
+(n-1)^2e_{n1}^2
 \Bigg\}
\nonumber
\\[1mm]
& &
\hspace{15mm}
+\, 
\frac{2p_n f_{n2}}{(p_n-1)n^2\sigma^2_n}  
\sum_{i,j=1}^{\ell-1} 
\Xb_{ni}\pr
(\mathbf{I}_{p_n}-\thetab_n \thetab_n\pr)
\Xb_{nj}
.
\label{sig2nell}
\end{eqnarray}
Therefore
\begin{equation}\label{espsig2nell}
{\rm E}[\sigma^2_{n\ell}]
=
\frac{2p_n \tilde{e}_{n2}}{n^2\sigma^2_n}  
\bigg\{ 
(\ell-1) 
\tilde{e}_{n2}
+
0
+(n-1)^2e_{n1}^2
 \bigg\}
+
\frac{2p_n (\ell-1) f_{n2}^2}{(p_n-1)n^2\sigma^2_n} ,
\end{equation}
where we have used Lemma~\ref{moments2}(iii).
This yields
\begin{equation}
\label{defsn}
s_n^2
:=
\sum_{\ell=1}^n
{\rm E}[\sigma^2_{n\ell}]
=
\frac{(n-1)p_n \tilde{e}_{n2}^2}{n\sigma^2_n}  
+
\frac{2p_n \tilde{e}_{n2}}{n\sigma^2_n}  
(n-1)^2e_{n1}^2
+
\frac{(n-1)p_n f_{n2}^2}{(p_n-1)n\sigma^2_n}  
\to 1 
\end{equation}
as~$n\to\infty$, as was to be shown.
\vspace{2mm}

(ii) From~(\ref{sig2nell}), we obtain
$
{\rm Var}
\big[
\sum_{\ell=1}^n \sigma^2_{n\ell}
\big]
\leq C
\big( 
{\rm Var}
\big[
A_{n}
\big]
+
{\rm Var}
\big[
B_{n}
\big]
+
{\rm Var}
\big[
C_{n}
\big]
\big)
,
$
where
$
A_{n}
:=
\frac{p_n \tilde{e}_{n2}}{n^2\sigma^2_n}  
\sum_{\ell=1}^n
\sum_{i,j=1}^{\ell-1} 
(\Xb_{ni}\pr\thetab_n-e_{n1})
(\Xb_{nj}\pr\thetab_n-e_{n1})
$, 
$B_{n}
:=
\frac{p_n e_{n1} \tilde{e}_{n2}}{n\sigma^2_n}  
\sum_{\ell=1}^n
\sum_{i=1}^{\ell-1} 
(\Xb_{ni}\pr\thetab_n-e_{n1})
$
and
$
C_n 
:=
\frac{f_{n2}}{n^2\sigma^2_n}  
\sum_{\ell=1}^n
\sum_{i,j=1}^{\ell-1} 
\Xb_{ni}\pr
(\mathbf{I}_{p_n}-\thetab_n \thetab_n\pr)
\Xb_{nj}
. 
$ 
We establish the result by showing that, under the assumptions considered, ${\rm Var}[A_n]$, ${\rm Var}[B_n]$ and ${\rm Var}[C_n]$ all are $o(1)$ as~$\ny$. 
We start with~$A_n$, which we split into
\begin{eqnarray*}
A_n
&\!\!\!=\!\!\!&
\frac{p_n \tilde{e}_{n2}}{n^2\sigma^2_n}  
\sum_{\ell=1}^n
\sum_{i=1}^{\ell-1} 
(\Xb_{ni}\pr\thetab_n-e_{n1})^2
+
\frac{2p_n \tilde{e}_{n2}}{n^2\sigma^2_n}  
\sum_{\ell=1}^n
\sum_{1\leq i<j\leq \ell-1} 
(\Xb_{ni}\pr\thetab_n-e_{n1})
(\Xb_{nj}\pr\thetab_n-e_{n1})
\\[2mm]
&\!\!\!=\!\!\!&
\frac{p_n \tilde{e}_{n2}}{n^2\sigma^2_n}  
\sum_{i=1}^{n-1}
\,(n-i) 
(\Xb_{ni}\pr\thetab_n-e_{n1})^2
+
\frac{2p_n \tilde{e}_{n2}}{n^2\sigma^2_n}  
\sum_{1\leq i<j\leq n-1} 
(n-j)
(\Xb_{ni}\pr\thetab_n-e_{n1})
(\Xb_{nj}\pr\thetab_n-e_{n1})
,
\end{eqnarray*}
that is, into~$A_{n}^{(1)}+A_{n}^{(2)}$, say. Clearly,
\begin{eqnarray*}
\lefteqn{
\hspace{-10mm}
{\rm Var}
\big[
A_{n}^{(1)}
\big]
=
\frac{p_n^2 \tilde{e}_{n2}^2}{n^4\sigma^4_n}  
\,
\sum_{i=1}^{n-1}
\,
(n-i)^2
\,
{\rm Var}
\big[
(\Xb_{ni}\pr\thetab_n-e_{n1})^2
\big]
}
\\[3mm]
& & 
\hspace{15mm}
\leq 
C\,
\frac{p_n^2 \tilde{e}_{n2}^2\big(\tilde{e}_{n4}-\tilde{e}_{n2}^2\big)}{n \sigma^4_n} 
\leq 
C\,
\frac{p_n^2 \tilde{e}_{n2}^2\big(\tilde{e}_{n4}-\tilde{e}_{n2}^2\big)}{n (p_n \tilde{e}_{n2}^2)^2} 
=
C\,
\Big(
\frac{\tilde{e}_{n4}}{n \tilde{e}_{n2}^2} 
-\frac{1}{n}
\Big),
\end{eqnarray*}
which, by assumption, is $o(1)$ as~$n\to\infty$. Since $(\Xb_{ni}\pr\thetab_n-e_{n1})
(\Xb_{nj}\pr\thetab_n-e_{n1})$, $i<j$, and $(\Xb_{nk}\pr\thetab_n-e_{n1})
(\Xb_{n\ell}\pr\thetab_n-e_{n1})$, $k<\ell$, are uncorrelated as soon as $(i,j)\neq(k,\ell)$, we obtain
$$
{\rm Var}
\big[
A_{n}^{(2)}
\big]
=
\frac{4p^2_n \tilde{e}_{n2}^2}{n^4\sigma^4_n}  
\sum_{1\leq i<j\leq n-1} 
(n-j)^2
\,
{\rm Var}
\big[
(\Xb_{ni}\pr\thetab_n-e_{n1})
(\Xb_{nj}\pr\thetab_n-e_{n1})
\big]
\leq
C
\,
\frac{p^2_n \tilde{e}_{n2}^4}{\sigma^4_n}  
\cdot
$$
In view of the majorations
$$
\frac{p^2_n \tilde{e}_{n2}^4}{\sigma^4_n}
\leq
C
\frac{p^2_n \tilde{e}_{n2}^4}{(2n p_n e_{n1}^2 \tilde{e}_{n2})^2}  
=
C
\Big(
\frac{\tilde{e}_{n2}}{n e_{n1}^2}  
\Big)^2
\quad
\textrm{and}
\quad
\frac{p^2_n \tilde{e}_{n2}^4}{\sigma^4_n}
\leq
C\Big(\frac{p_n \tilde{e}_{n2}^2}{f_{n2}^2} \Big)^2 
,
$$
${\rm Var}
\big[
A_{n}^{(2)}
\big]
$, by assumption, is $o(1)$ as~$n\to\infty$. Therefore, ${\rm Var}[A_n]$ is indeed $o(1)$ as~$n\to\infty$.

Turning to~$B_n$, 
$$
{\rm Var}[B_n]
=
\frac{p_n^2 e_{n1}^2 \tilde{e}_{n2}^2}{n^2\sigma^4_n}  
{\rm Var}\Big[
\sum_{i=1}^{n-1} (n-i)
(\Xb_{ni}\pr\thetab_n-e_{n1})
\Big]
=
\frac{p_n^2 e_{n1}^2 \tilde{e}_{n2}^2}{n^2\sigma^4_n}  
\sum_{i=1}^{n-1} (n-i)^2
\tilde{e}_{n2}
\leq 
C 
\frac{n p_n^2 e_{n1}^2 \tilde{e}_{n2}^3}{\sigma^4_n}  
,
$$
which is $o(1)$ as~$n\to\infty$ since it can be upper-bounded by
$$
C 
\frac{n p_n^2 e_{n1}^2 \tilde{e}_{n2}^3}{(2n p_n e_{n1}^2 \tilde{e}_{n2})^2}  
=
C
\frac{\tilde{e}_{n2}}{n e_{n1}^2}  
\quad
\textrm{and by}
\quad
C 
\frac{n p_n^2 e_{n1}^2 \tilde{e}_{n2}^3}{n p_n e_{n1}^2 \tilde{e}_{n2} f_{n2}^2} 
=
C 
\frac{p_n \tilde{e}_{n2}^2}{f_{n2}^2} 
\cdot 
$$

Finally, we consider~$C_n$. Proceeding as for~$A_n$, we split $C_n$ into
\begin{eqnarray*}
C_n
&\!\!\!=\!\!\!&
\frac{f_{n2}}{n^2\sigma^2_n}  
\sum_{\ell=1}^n
\sum_{i=1}^{\ell-1} 
\Xb_{ni}\pr
(\mathbf{I}_{p_n}-\thetab_n \thetab_n\pr)
\Xb_{ni}
+
\frac{2f_{n2}}{n^2\sigma^2_n}  
\sum_{\ell=1}^n
\sum_{1\leq i<j\leq \ell-1} 
\Xb_{ni}\pr
(\mathbf{I}_{p_n}-\thetab_n \thetab_n\pr)
\Xb_{nj}
\\[2mm]
&\!\!\!=\!\!\!&
\frac{f_{n2}}{n^2\sigma^2_n}  
\sum_{i=1}^{n-1}
\,(n-i) 
\Xb_{ni}\pr
(\mathbf{I}_{p_n}-\thetab_n \thetab_n\pr)
\Xb_{ni}
+
\frac{2f_{n2}}{n^2\sigma^2_n}  
\sum_{1\leq i<j\leq n-1} 
(n-j)
\Xb_{ni}\pr
(\mathbf{I}_{p_n}-\thetab_n \thetab_n\pr)
\Xb_{nj}
,
\end{eqnarray*}
that is, into~$C_{n}^{(1)}+C_{n}^{(2)}$, say. 
Clearly,
$$
{\rm Var}
\big[
C_{n}^{(1)}
\big]
=
\frac{f_{n2}^2}{n^4\sigma^4_n}  
\,
\sum_{i=1}^{n-1}
\,
(n-i)^2
\,
{\rm Var}
\big[
\Xb_{ni}\pr
(\mathbf{I}_{p_n}-\thetab_n \thetab_n\pr)
\Xb_{ni}
\big]
\leq 
C
\frac{f_{n2}^2 (f_{n4}-f_{n2}^2)}{n\sigma^4_n}  
\leq 
C
\frac{f_{n4}-f_{n2}^2}{n f_{n2}^2}  
,
$$
so that ${\rm Var}\big[C_{n}^{(1)}\big]$ is $o(1)$ as~$n\to\infty$. 
Since $\Xb_{ni}\pr
(\mathbf{I}_{p_n}-\thetab_n \thetab_n\pr)
\Xb_{nj}
$, $i<j$, and $\Xb_{nk}\pr
(\mathbf{I}_{p_n}-\thetab_n \thetab_n\pr)
\Xb_{n\ell}
$, $k<\ell$, are uncorrelated as soon as $(i,j)\neq(k,\ell)$, we obtain
$$
{\rm Var}
\big[
C_{n}^{(2)}
\big]
=
\frac{4f_{n2}^2}{n^4\sigma^4_n}  
\sum_{1\leq i<j\leq n-1} 
(n-j)^2
\,
{\rm Var}
\big[
\Xb_{ni}\pr
(\mathbf{I}_{p_n}-\thetab_n \thetab_n\pr)
\Xb_{nj}
\big]
\leq
C
\frac{f_{n2}^4}{\sigma^4_n(p_n-1)}  
\leq
\frac{C}{p_n}  
\cdot
$$
Therefore, ${\rm Var}[C_n]$ is also $o(1)$ as~$n\to\infty$, which establishes the result.
\cqfd
\vspace{3mm}


{\sc Proof of Lemma~\ref{THElemmaspher}}.
the Cauchy-Schwarz and Chebychev inequalities yield
\begin{equation}
\label{termeLemme3}
\sum_{\ell=1}^n {\rm E}\big[D_{n\ell}^2 \mathbb{I}[|D_{n\ell}|>\varepsilon]\big]
\leq 
\sum_{\ell=1}^n \sqrt{{\rm E}\big[D_{n\ell}^4 \big]{\rm P}\big[| D_{n\ell}|>\varepsilon\big]}
\leq 
\frac{1}{\varepsilon}\sum_{\ell=1}^n \sqrt{{\rm E}\big[D_{n\ell}^4 \big]{\rm Var}\big[D_{n\ell}\big]}
.
\end{equation}

Recalling that $\sigma_{n\ell}^2={\rm E}\big[D_{n\ell}^2 \,|\,\mathcal{F}_{n,\ell-1}\big]$, (\ref{espsig2nell}) provides
$$
{\rm Var}[D_{n\ell}]
\leq 
{\rm E}\big[D_{n\ell}^2\big]
=
{\rm E}[\sigma_{n\ell}^2]
\leq 
\frac{2p_n}{n\sigma_n^2}
\bigg(
\tilde{e}_{n2}^2+ne_{n1}^2\tilde{e}_{n2}+\frac{f_{n2}^2}{p_n-1}
\bigg)
\leq 
\frac{C}{n}
\cdot
$$

Using~(\ref{Dnell}) and the inequalities $(a+b)^4\leq 8(a^4+b^4)$ and~$\sigma^2_n\geq 2n p_ne_{n1}^2 \tilde{e}_{n2}$ then yields
\begin{eqnarray}
\nonumber
{\rm E}\big[D_{n\ell}^4\big]
&\leq&
\frac{Cp_n^2}{n^4\sigma_n^4}
\Bigg(
{\rm E}
\bigg[
\bigg(
\sum_{i=1}^{\ell-1}
\,
(\Xb_{ni}-e_{n1}\thetab_n)'(\Xb_{n\ell}-e_{n1}\thetab_n)
\bigg)^4
\bigg]
+
n^4e_{n1}^4 {\rm E}\big[(\Xb_{n\ell}\pr \thetab_n-e_{n1})^4\big]
\Bigg)
\\[2mm]
&\leq &
\frac{Cp_n^2}{n^4\sigma_n^4}
\,
{\rm E}
\bigg[
\bigg(
\sum_{i=1}^{\ell-1}
\,
(\Xb_{ni}-e_{n1}\thetab_n)'(\Xb_{n\ell}-e_{n1}\thetab_n)
\bigg)^4
\bigg]
+
\frac{C \tilde{e}_{n4}}{n^2 \tilde{e}_{n2}^2}
.
\label{borneDnell4}
\end{eqnarray}
Applying Lemma~\ref{moments4}, we have
\begin{multline*}
{\rm E}
\bigg[
\bigg(
\sum_{i=1}^{\ell-1}
\,
(\Xb_{ni}-e_{n1}\thetab_n)'(\Xb_{n\ell}-e_{n1}\thetab_n)
\bigg)^4
\bigg]
=
(\ell-1)
\left(
\tilde{e}_{n4}^2
+
\frac{6}{p_n-1}{\rm E}\big[v_{ni}^2(u_{ni}-e_{n1})^2\big]^2
+
\frac{3f_{n4}^2}{p_n^2-1}
\right)
\\[2mm]
+
3(\ell-1)(\ell-2)
\left(
\tilde{e}_{n2}^2\tilde{e}_{n4}
+
\frac{2\tilde{e}_{n2}f_{n2}}{p_n-1} {\rm E}\big[v_{ni}^2(u_{ni}-e_{n1})^2\big]
+
\frac{f_{n2}^2f_{n4}}{(p_n-1)^2} 
\right),
\end{multline*}
By Cauchy-Schwarz, this yields
\begin{eqnarray*}
\lefteqn{
\frac{p_n^2}{n^4\sigma_n^4}
\,
{\rm E}
\bigg[
\bigg(
\sum_{i=1}^{\ell-1}
\,
(\Xb_{ni}-e_{n1}\thetab_n)'(\Xb_{n\ell}-e_{n1}\thetab_n)
\bigg)^4
\bigg]
}
\\[2mm]
& & 
\hspace{2mm}
\leq
\frac{1}{n^3\sigma_n^4}
\left( 
p_n^2\tilde{e}_{n4}^2
+
 6p_nf_{n4}\tilde{e}_{n4}
+
3f_{n4}^2
\right)
+
\frac{3}{n^2\sigma_n^4}
\left( 
 p_n^2 \tilde{e}_{n2}^2\tilde{e}_{n4}
+
 2p_n \tilde{e}_{n2}f_{n2}f_{n4}^{1/2}\tilde{e}_{n4}^{1/2} 
+
f_{n2}^2f_{n4} 
\right)
\\[2mm]
& & 
\hspace{2mm}
\leq
\frac{C}{n^3}
\bigg( 
\frac{\tilde{e}_{n4}^2}{\tilde{e}_{n2}^4}
+
\frac{f_{n4}\tilde{e}_{n4}}{f_{n2}^2\tilde{e}_{n2}^2}
+
\frac{f_{n4}^2}{f_{n2}^4}
\bigg)
+
\frac{C}{n^2}
\bigg( 
\frac{\tilde{e}_{n4}}{\tilde{e}_{n2}^2}
+ 
\Big(\frac{f_{n4}\tilde{e}_{n4}}{f_{n2}^2\tilde{e}_{n2}^2}\Big)^{1/2}
+
\frac{f_{n4}}{f_{n2}^2}
\bigg)
.
\end{eqnarray*}
Plugging into~(\ref{borneDnell4}), we conclude that 
$$
{\rm E}\big[D_{n\ell}^4\big]
\leq
\frac{C}{n^3}
\bigg( 
\frac{\tilde{e}_{n4}^2}{\tilde{e}_{n2}^4}
+
\frac{f_{n4}\tilde{e}_{n4}}{f_{n2}^2\tilde{e}_{n2}^2}
+
\frac{f_{n4}^2}{f_{n2}^4}
\bigg)
+
\frac{C}{n^2}
\bigg( 
\frac{\tilde{e}_{n4}}{\tilde{e}_{n2}^2}
+ 
\Big(\frac{f_{n4}\tilde{e}_{n4}}{f_{n2}^2\tilde{e}_{n2}^2}\Big)^{1/2}
+
\frac{f_{n4}}{f_{n2}^2}
\bigg)
\leq
\frac{C}{n}
\bigg( 
\frac{\tilde{e}_{n4}}{n\tilde{e}_{n2}^2}
+
\frac{f_{n4}}{nf_{n2}^2}
\bigg)
,
$$
which, by assumption, is $o(1/n)$ as~$\ny$. 

All majorations and $o$'s above being uniform in~$\ell$, we finally obtain that 
$$
\sum_{\ell=1}^n \sqrt{{\rm E}\big[D_{n\ell}^4 \big]{\rm Var}\big[D_{n\ell}\big]}
\leq
C \Big( n \max_{\ell=1,\ldots,n} {\rm E}\big[D_{n\ell}^4 \big] \Big)^{1/2}
\to 0
$$
as~$\ny$, which, in view of~(\ref{termeLemme3}), establishes the result.
\cqfd


\vspace{2mm}

%

{\sc Proof of Theorem~\ref{Powerprop}}.
From Theorem~\ref{maintheorempower}, we have that, as~$\ny$,
\begin{eqnarray*}
\lefteqn{
\bigg|
{\rm P}\n_{F_n}[R_n^{\rm St} >  z_\alpha]
-
\Big(
1- 
\Phi\Big(\! z_\alpha- \frac{\tau^2}{\sqrt{2}}\Big)
\Big)
\bigg| 
=
\bigg|
{\rm P}\n_{F_n}[R_n^{\rm St} \leq  z_\alpha]
-
\Phi\Big(\! z_\alpha- \frac{\tau^2}{\sqrt{2}}\Big)
\bigg| 
}
\nonumber
\\[3mm]
& & 
\hspace{0mm}
\leq
\sup_{z\in \R}
\bigg|
{\rm P}\n_{F_n}\Big[
\frac{R_n^{\rm St}-{\rm E}[{R}_n^{\rm St}]}{\sigma_n}
\leq 
z
\Big]
-
\Phi
(z)\bigg|
+
\bigg|
\Phi
\Big(
\frac{z_\alpha-{\rm E}[{R}_n^{\rm St}]}{\sigma_n}
\Big)
-
\Phi\bigg(\! z_\alpha- \frac{\tau^2}{\sqrt{2}}\bigg)
\bigg| 
\to 0
,
\end{eqnarray*}
where we used Lemma~2.11 from~\cite{van1998}.
\cqfd


\section*{Acknowledgements}

We would like to thank the Associate Editor and three anonymous referees for their insightful comments and suggestions, that led to a substantial improvement of a previous version of this work. We are particularly grateful to the referee who opened a new field of research to us by encouraging us to look at the results from a minimax separation rate point of view.


\begin{supplement}[id=suppA]
  \stitle{Supplement to ``Testing Uniformity on High-Dimensional Spheres against Rotationally Symmetric Alternatives"}
  \slink[doi]{completed by the typesetter}
  \sdatatype{.pdf}
  \sdescription{In this supplementary article, we derive the \mbox{fixed-$p$} asymptotic non-null distribution of the Rayleigh test statistic in~(\ref{fixedplaw}), and we show that,  under FvML distributions, the conditions~(i)-(iii) of Theorem~\ref{maintheorempower} always hold.}
\end{supplement}

\bibliographystyle{imsart-nameyear.bst}
\bibliography{hdLAN.bib}           

\providecommand{\noopsort}[1]{}
\begin{thebibliography}{42}

\bibitem[\protect\citeauthoryear{Azzalini and Capitanio}{1999}]{AzzCap1999}
\begin{barticle}[author]
\bauthor{\bsnm{Azzalini},~\bfnm{A.}\binits{A.}} \AND
  \bauthor{\bsnm{Capitanio},~\bfnm{A.}\binits{A.}}
(\byear{1999}).
\btitle{Statistical applications of the multivariate skew normal distribution}.
\bjournal{J. R. Stat. Soc. Ser. B}
\bvolume{61}
\bpages{579--602}.
\end{barticle}
\endbibitem

\bibitem[\protect\citeauthoryear{Banerjee and Ghosh}{2004}]{BanGho2004}
\begin{barticle}[author]
\bauthor{\bsnm{Banerjee},~\bfnm{A.}\binits{A.}} \AND
  \bauthor{\bsnm{Ghosh},~\bfnm{J.}\binits{J.}}
(\byear{2004}).
\btitle{Frequency sensitive competitive learning for scalable balanced
  clustering on high-dimensional hyperspheres}.
\bjournal{IEEE T. Neural Networ.}
\bvolume{15}
\bpages{702--719}.
\end{barticle}
\endbibitem

\bibitem[\protect\citeauthoryear{Banerjee
  et~al.}{2003}]{banerjee2003generative}
\begin{barticle}[author]
\bauthor{\bsnm{Banerjee},~\bfnm{Arindam}\binits{A.}},
  \bauthor{\bsnm{Dhillon},~\bfnm{Inderjit}\binits{I.}},
  \bauthor{\bsnm{Ghosh},~\bfnm{Joydeep}\binits{J.}} \AND
  \bauthor{\bsnm{Sra},~\bfnm{Suvrit}\binits{S.}}
(\byear{2003}).
\btitle{Generative model-based clustering of directional data}.
\bjournal{{\rm In} Proceedings of the ninth ACM SIGKDD international conference
  on Knowledge discovery and data mining}
\bvolume{\hspace{-1mm}}
\bpages{19--28}.
\end{barticle}
\endbibitem

\bibitem[\protect\citeauthoryear{Bickel et~al.}{1998}]{Bic1998}
\begin{bbook}[author]
\bauthor{\bsnm{Bickel},~\bfnm{Peter~J.}\binits{P.~J.}},
  \bauthor{\bsnm{Klaassen},~\bfnm{Chris A.~J.}\binits{C.~A.~J.}},
  \bauthor{\bsnm{Ritov},~\bfnm{Ya'acov}\binits{Y.}} \AND
  \bauthor{\bsnm{Wellner},~\bfnm{Jon~A.}\binits{J.~A.}}
(\byear{1998}).
\btitle{Efficient and Adaptive Estimation for Semiparametric Models}.
\bpublisher{Springer}, \baddress{New York}.
\end{bbook}
\endbibitem

\bibitem[\protect\citeauthoryear{Cai, Fan and Jiang}{2013}]{Caietal2013}
\begin{barticle}[author]
\bauthor{\bsnm{Cai},~\bfnm{Tony}\binits{T.}},
  \bauthor{\bsnm{Fan},~\bfnm{Jianqing}\binits{J.}} \AND
  \bauthor{\bsnm{Jiang},~\bfnm{Tiefeng}\binits{T.}}
(\byear{2013}).
\btitle{Distributions of angles in random packing on spheres}.
\bjournal{J. Mach. Learn. Res.}
\bvolume{14}
\bpages{1837--1864}.
\end{barticle}
\endbibitem

\bibitem[\protect\citeauthoryear{Chaudhuri}{1992}]{Cha1992}
\begin{barticle}[author]
\bauthor{\bsnm{Chaudhuri},~\bfnm{P.}\binits{P.}}
(\byear{1992}).
\btitle{Multivariate location estimation using extension of R-estimates through
  U-statistics type approach}.
\bjournal{Ann. Statist.}
\bvolume{20}
\bpages{897--916}.
\end{barticle}
\endbibitem

\bibitem[\protect\citeauthoryear{Chikuse}{1991}]{Chi1991}
\begin{barticle}[author]
\bauthor{\bsnm{Chikuse},~\bfnm{Yasuko}\binits{Y.}}
(\byear{1991}).
\btitle{High dimensional limit theorems and matrix decompositions on the
  Stiefel manifold}.
\bjournal{J. Multivariate anal.}
\bvolume{36}
\bpages{145--162}.
\end{barticle}
\endbibitem

\bibitem[\protect\citeauthoryear{Chikuse}{1993}]{Chi1993}
\begin{barticle}[author]
\bauthor{\bsnm{Chikuse},~\bfnm{Yasuko}\binits{Y.}}
(\byear{1993}).
\btitle{High dimensional asymptotic expansions for the matrix Langevin
  distributions on the Stiefel manifold}.
\bjournal{J. Multivariate Anal.}
\bvolume{44}
\bpages{82--101}.
\end{barticle}
\endbibitem

\bibitem[\protect\citeauthoryear{Chikuse}{2003}]{Chi2003}
\begin{bbook}[author]
\bauthor{\bsnm{Chikuse},~\bfnm{Yasuko}\binits{Y.}}
(\byear{2003}).
\btitle{Statistics on Special Manifolds}.
\bseries{Lecture Notes in Statistics}
\bvolume{174}.
\bpublisher{Springer}, \baddress{New York}.
\end{bbook}
\endbibitem

\bibitem[\protect\citeauthoryear{Cuesta-Albertos, Cuevas and
  Fraiman}{2009}]{cueetal2009}
\begin{barticle}[author]
\bauthor{\bsnm{Cuesta-Albertos},~\bfnm{Juan~A.}\binits{J.~A.}},
  \bauthor{\bsnm{Cuevas},~\bfnm{Antonio}\binits{A.}} \AND
  \bauthor{\bsnm{Fraiman},~\bfnm{Ricardo}\binits{R.}}
(\byear{2009}).
\btitle{On projection-based tests for directional and compositional data}.
\bjournal{Stat. Comput.}
\bvolume{19}
\bpages{367--380}.
\end{barticle}
\endbibitem

\bibitem[\protect\citeauthoryear{Cutting, Paindaveine and
  Verdebout}{2015}]{Cut2015}
\begin{barticle}[author]
\bauthor{\bsnm{Cutting},~\bfnm{Christine}\binits{C.}},
  \bauthor{\bsnm{Paindaveine},~\bfnm{Davy}\binits{D.}} \AND
  \bauthor{\bsnm{Verdebout},~\bfnm{Thomas}\binits{T.}}
(\byear{2015}).
\btitle{Supplement to ``Testing uniformity on high-dimensional spheres against
  rotationally symmetric alternatives"}.
\bpages{Submitted}.
\end{barticle}
\endbibitem

\bibitem[\protect\citeauthoryear{Davies}{1977}]{Dav1977}
\begin{barticle}[author]
\bauthor{\bsnm{Davies},~\bfnm{Robert~B.}\binits{R.~B.}}
(\byear{1977}).
\btitle{Hypothesis testing when a nuisance parameter is present only under the
  alternative}.
\bjournal{Biometrika}
\bvolume{64}
\bpages{247--254}.
\end{barticle}
\endbibitem

\bibitem[\protect\citeauthoryear{Davies}{1987}]{Dav1987}
\begin{barticle}[author]
\bauthor{\bsnm{Davies},~\bfnm{Robert~B.}\binits{R.~B.}}
(\byear{1987}).
\btitle{Hypothesis testing when a nuisance parameter is present only under the
  alternative}.
\bjournal{Biometrika}
\bvolume{74}
\bpages{33-43}.
\end{barticle}
\endbibitem

\bibitem[\protect\citeauthoryear{Davies}{2002}]{Dav2002}
\begin{barticle}[author]
\bauthor{\bsnm{Davies},~\bfnm{Robert~B.}\binits{R.~B.}}
(\byear{2002}).
\btitle{Hypothesis testing when a nuisance parameter is present only under the
  alternative: Linear model case}.
\bjournal{Biometrika}
\bvolume{89}
\bpages{484--489}.
\end{barticle}
\endbibitem

\bibitem[\protect\citeauthoryear{Dryden}{2005}]{Dry2005}
\begin{barticle}[author]
\bauthor{\bsnm{Dryden},~\bfnm{I.~L.}\binits{I.~L.}}
(\byear{2005}).
\btitle{Statistical analysis on high-dimensional spheres and shape spaces}.
\bjournal{Ann. Statist.}
\bvolume{33}
\bpages{1643--1665}.
\end{barticle}
\endbibitem

\bibitem[\protect\citeauthoryear{Eisen et~al.}{1998}]{Eisetal1998}
\begin{barticle}[author]
\bauthor{\bsnm{Eisen},~\bfnm{M.~B.}\binits{M.~B.}},
  \bauthor{\bsnm{Spellman},~\bfnm{P.~T.}\binits{P.~T.}},
  \bauthor{\bsnm{Brown},~\bfnm{P.~O.}\binits{P.~O.}} \AND
  \bauthor{\bsnm{Botstein},~\bfnm{D.}\binits{D.}}
(\byear{1998}).
\btitle{Cluster analysis and display of genome-wide expression patterns}.
\bjournal{Proc. Natl. Acad. Sci. USA}
\bvolume{95}
\bpages{14863--14868}.
\end{barticle}
\endbibitem

\bibitem[\protect\citeauthoryear{Fa{\"y} et~al.}{2013}]{Fayetal2013}
\begin{barticle}[author]
\bauthor{\bsnm{Fa{\"y}},~\bfnm{Gilles}\binits{G.}},
  \bauthor{\bsnm{Delabrouille},~\bfnm{Jacques}\binits{J.}},
  \bauthor{\bsnm{Kerkyacharian},~\bfnm{G\'{e}rard}\binits{G.}} \AND
  \bauthor{\bsnm{Picard},~\bfnm{Dominique}\binits{D.}}
(\byear{2013}).
\btitle{Testing the isotropy of high energy cosmic rays using spherical
  needlets}.
\bjournal{Ann. Appl. Stat.}
\bvolume{7}
\bpages{1040-1073}.
\end{barticle}
\endbibitem

\bibitem[\protect\citeauthoryear{Giri}{1996}]{Gir1996}
\begin{bbook}[author]
\bauthor{\bsnm{Giri},~\bfnm{Narayan~C.}\binits{N.~C.}}
(\byear{1996}).
\btitle{Group Invariance in Statistical Inference}.
\bpublisher{World Scientific Publishing Company}, \baddress{Singapore}.
\end{bbook}
\endbibitem

\bibitem[\protect\citeauthoryear{Hallin and Paindaveine}{2006}]{HP06}
\begin{barticle}[author]
\bauthor{\bsnm{Hallin},~\bfnm{Marc}\binits{M.}} \AND
  \bauthor{\bsnm{Paindaveine},~\bfnm{Davy}\binits{D.}}
(\byear{2006}).
\btitle{Semiparametrically efficient rank-based inference for shape. I. Optimal
  rank-based tests for sphericity}.
\bjournal{Ann. Statist.}
\bvolume{34}
\bpages{2707--2756}.
\end{barticle}
\endbibitem

\bibitem[\protect\citeauthoryear{Hornika and Gr\"{u}n}{2014}]{HotGru2014}
\begin{barticle}[author]
\bauthor{\bsnm{Hornika},~\bfnm{Kurt}\binits{K.}} \AND
  \bauthor{\bsnm{Gr\"{u}n},~\bfnm{Bettina}\binits{B.}}
(\byear{2014}).
\btitle{movMF: An R package for fitting mixtures of von Mises-Fisher
  distributions}.
\bjournal{J. Statist. Softw.}
\bvolume{58}.
\end{barticle}
\endbibitem

\bibitem[\protect\citeauthoryear{Ingster}{2000}]{Ing2000}
\begin{barticle}[author]
\bauthor{\bsnm{Ingster},~\bfnm{Y.~I.}\binits{Y.~I.}}
(\byear{2000}).
\btitle{Adaptive chi-square tests}.
\bjournal{J. Math. Sciences}
\bvolume{99}
\bpages{1110--1120}.
\end{barticle}
\endbibitem

\bibitem[\protect\citeauthoryear{John}{1972}]{Joh1972}
\begin{barticle}[author]
\bauthor{\bsnm{John},~\bfnm{S.}\binits{S.}}
(\byear{1972}).
\btitle{The distribution of a statistic used for testing sphericity of normal
  distributions}.
\bjournal{Biometrika}
\bvolume{59}
\bpages{169--173}.
\end{barticle}
\endbibitem

\bibitem[\protect\citeauthoryear{Juan and Prieto}{2001}]{JuaPri2001}
\begin{barticle}[author]
\bauthor{\bsnm{Juan},~\bfnm{Jesus}\binits{J.}} \AND
  \bauthor{\bsnm{Prieto},~\bfnm{Francisco~J.}\binits{F.~J.}}
(\byear{2001}).
\btitle{Using angles to identify concentrated multivariate outliers}.
\bjournal{Technometrics}
\bvolume{43}
\bpages{311--322}.
\end{barticle}
\endbibitem

\bibitem[\protect\citeauthoryear{Kim, Koo and Ngoc}{2016}]{TM15}
\begin{barticle}[author]
\bauthor{\bsnm{Kim},~\bfnm{Peter}\binits{P.}},
  \bauthor{\bsnm{Koo},~\bfnm{Ja-Yong}\binits{J.-Y.}} \AND
  \bauthor{\bsnm{Ngoc},~\bfnm{Thanh Mai~Pham}\binits{T.~M.~P.}}
(\byear{2016}).
\btitle{Supersmooth Testing on the Sphere over Analytic Classes}.
\bjournal{J. Nonparam. Stat.}
\bpages{to appear}.
\end{barticle}
\endbibitem

\bibitem[\protect\citeauthoryear{Lacour and Ngoc}{2014}]{TM14}
\begin{barticle}[author]
\bauthor{\bsnm{Lacour},~\bfnm{Claire}\binits{C.}} \AND
  \bauthor{\bsnm{Ngoc},~\bfnm{Thanh Mai~Pham}\binits{T.~M.~P.}}
(\byear{2014}).
\btitle{Goodness-of-fit test for noisy directional data}.
\bjournal{Bernoulli}
\bvolume{20}
\bpages{2131--2168}.
\end{barticle}
\endbibitem

\bibitem[\protect\citeauthoryear{Ledoit and Wolf}{2002}]{LedWol2002}
\begin{barticle}[author]
\bauthor{\bsnm{Ledoit},~\bfnm{Olivier}\binits{O.}} \AND
  \bauthor{\bsnm{Wolf},~\bfnm{Michael}\binits{M.}}
(\byear{2002}).
\btitle{Some hypothesis tests for the covariance matrix when the dimension is
  large compared to the sample size}.
\bjournal{Ann. Statist.}
\bvolume{30}
\bpages{1081--1102}.
\end{barticle}
\endbibitem

\bibitem[\protect\citeauthoryear{Lehmann and Romano}{2005}]{Lehetal2005}
\begin{bbook}[author]
\bauthor{\bsnm{Lehmann},~\bfnm{E.~L.}\binits{E.~L.}} \AND
  \bauthor{\bsnm{Romano},~\bfnm{J.~P.}\binits{J.~P.}}
(\byear{2005}).
\btitle{Testing Statistical Hypotheses}.
\bpublisher{Springer}, \baddress{New York}.
\end{bbook}
\endbibitem

\bibitem[\protect\citeauthoryear{Liese and Miescke}{2008}]{Lie2008}
\begin{bbook}[author]
\bauthor{\bsnm{Liese},~\bfnm{Friedrich}\binits{F.}} \AND
  \bauthor{\bsnm{Miescke},~\bfnm{Klaus-J.}\binits{K.-J.}}
(\byear{2008}).
\btitle{Statistical Decision Theory: Estimation, Testing, and Selection}.
\bpublisher{Springer}, \baddress{New York}.
\end{bbook}
\endbibitem

\bibitem[\protect\citeauthoryear{Mardia and Jupp}{2000}]{MarJup2000}
\begin{bbook}[author]
\bauthor{\bsnm{Mardia},~\bfnm{Kanti~V.}\binits{K.~V.}} \AND
  \bauthor{\bsnm{Jupp},~\bfnm{Peter~E.}\binits{P.~E.}}
(\byear{2000}).
\btitle{Directional Statistics}.
\bpublisher{John Wiley \& Sons}, \baddress{Chichester}.
\end{bbook}
\endbibitem

\bibitem[\protect\citeauthoryear{Moreira}{2009}]{Mor2009}
\begin{barticle}[author]
\bauthor{\bsnm{Moreira},~\bfnm{M.~J.}\binits{M.~J.}}
(\byear{2009}).
\btitle{A maximum likelihood method for the incidental parameter problem}.
\bjournal{Ann. Statist.}
\bvolume{37}
\bpages{3660--3696}.
\end{barticle}
\endbibitem

\bibitem[\protect\citeauthoryear{M{\"o}tt{\"o}nen and Oja}{1995}]{mooj95}
\begin{barticle}[author]
\bauthor{\bsnm{M{\"o}tt{\"o}nen},~\bfnm{Jyrki}\binits{J.}} \AND
  \bauthor{\bsnm{Oja},~\bfnm{Hannu}\binits{H.}}
(\byear{1995}).
\btitle{Multivariate spatial sign and rank methods}.
\bjournal{J. Nonparametric Stat.}
\bvolume{5}
\bpages{201--213}.
\end{barticle}
\endbibitem

\bibitem[\protect\citeauthoryear{Oja}{2010}]{Oja2010}
\begin{bbook}[author]
\bauthor{\bsnm{Oja},~\bfnm{Hannu}\binits{H.}}
(\byear{2010}).
\btitle{Multivariate Nonparametric Methods with R. An Approach Based on Spatial
  Signs and Ranks}.
\bpublisher{Springer-Verlag}, \baddress{New York}.
\end{bbook}
\endbibitem

\bibitem[\protect\citeauthoryear{Onatski, Moreira and
  Hallin}{2013}]{Onaetal2013}
\begin{barticle}[author]
\bauthor{\bsnm{Onatski},~\bfnm{A.}\binits{A.}},
  \bauthor{\bsnm{Moreira},~\bfnm{M.~J.}\binits{M.~J.}} \AND
  \bauthor{\bsnm{Hallin},~\bfnm{M.}\binits{M.}}
(\byear{2013}).
\btitle{Asymptotic power of sphericity tests for high-dimensional data}.
\bjournal{Ann. Statist.}
\bvolume{41}
\bpages{1204--1231}.
\end{barticle}
\endbibitem

\bibitem[\protect\citeauthoryear{Onatski, Moreira and
  Hallin}{2014}]{Onaetal2014}
\begin{barticle}[author]
\bauthor{\bsnm{Onatski},~\bfnm{A.}\binits{A.}},
  \bauthor{\bsnm{Moreira},~\bfnm{M.~J.}\binits{M.~J.}} \AND
  \bauthor{\bsnm{Hallin},~\bfnm{M.}\binits{M.}}
(\byear{2014}).
\btitle{Signal detection in high dimension: the multispiked case}.
\bjournal{Ann. Statist.}
\bvolume{42}
\bpages{225--254}.
\end{barticle}
\endbibitem

\bibitem[\protect\citeauthoryear{Paindaveine and Verdebout}{2015}]{PaiVer2015}
\begin{barticle}[author]
\bauthor{\bsnm{Paindaveine},~\bfnm{Davy}\binits{D.}} \AND
  \bauthor{\bsnm{Verdebout},~\bfnm{Thomas}\binits{T.}}
(\byear{2015}).
\btitle{On high-dimensional sign tests}.
\bjournal{Bernoulli,}
\bpages{to appear}.
\end{barticle}
\endbibitem

\bibitem[\protect\citeauthoryear{Rayleigh}{1919}]{Ray1919}
\begin{barticle}[author]
\bauthor{\bsnm{Rayleigh},~\bfnm{Lord}\binits{L.}}
(\byear{1919}).
\btitle{On the problem of random vibrations and random flights in one, two and
  three dimensions}.
\bjournal{Phil. Mag.}
\bvolume{37}
\bpages{321--346}.
\end{barticle}
\endbibitem

\bibitem[\protect\citeauthoryear{Saw}{1978}]{saw1978}
\begin{barticle}[author]
\bauthor{\bsnm{Saw},~\bfnm{J.~G.}\binits{J.~G.}}
(\byear{1978}).
\btitle{A family of distributions on the $m$-sphere and some hypothesis tests}.
\bjournal{Biometrika}
\bvolume{65}
\bpages{69--73}.
\end{barticle}
\endbibitem

\bibitem[\protect\citeauthoryear{Shao}{2003}]{Shao2003}
\begin{bbook}[author]
\bauthor{\bsnm{Shao},~\bfnm{J.}\binits{J.}}
(\byear{2003}).
\btitle{Mathematical Statistics}.
\bpublisher{Springer}, \baddress{New York}.
\end{bbook}
\endbibitem

\bibitem[\protect\citeauthoryear{van~der Vaart}{1998}]{van1998}
\begin{bbook}[author]
\bauthor{\bparticle{van~der} \bsnm{Vaart},~\bfnm{A.~W.}\binits{A.~W.}}
(\byear{1998}).
\btitle{Asymptotic Statistics}.
\bpublisher{Cambridge Univ. Press}, \baddress{Cambridge}.
\end{bbook}
\endbibitem

\bibitem[\protect\citeauthoryear{Wang, Peng and Li}{2015}]{Runze2015}
\begin{barticle}[author]
\bauthor{\bsnm{Wang},~\bfnm{Lan}\binits{L.}},
  \bauthor{\bsnm{Peng},~\bfnm{Bo}\binits{B.}} \AND
  \bauthor{\bsnm{Li},~\bfnm{Runze}\binits{R.}}
(\byear{2015}).
\btitle{A high-dimensional nonparametric multivariate test for mean vector}.
\bjournal{J. Amer. Statist. Assoc.,}
\bpages{to appear}.
\end{barticle}
\endbibitem

\bibitem[\protect\citeauthoryear{Watson}{1988}]{Wat1988}
\begin{barticle}[author]
\bauthor{\bsnm{Watson},~\bfnm{G.~S.}\binits{G.~S.}}
(\byear{1988}).
\btitle{The Langevin distribution on high dimensional spheres}.
\bjournal{J. Appl. Statist.}
\bvolume{15}
\bpages{123--130}.
\end{barticle}
\endbibitem

\bibitem[\protect\citeauthoryear{Zou et~al.}{2014}]{Zouetal2013}
\begin{barticle}[author]
\bauthor{\bsnm{Zou},~\bfnm{Changliang}\binits{C.}},
  \bauthor{\bsnm{Peng},~\bfnm{Liuhua}\binits{L.}},
  \bauthor{\bsnm{Feng},~\bfnm{Long}\binits{L.}} \AND
  \bauthor{\bsnm{Wang},~\bfnm{Zhaojun}\binits{Z.}}
(\byear{2014}).
\btitle{Multivariate-sign-based high-dimensional tests for sphericity}.
\bjournal{Biometrika}
\bvolume{101}
\bpages{229--236}.
\end{barticle}
\endbibitem

\end{thebibliography}
\vspace{3mm} 


\end{document}